\documentclass[11pt]{article}

\usepackage{amssymb}
\usepackage{amsthm}
\usepackage{graphicx, gastex}
\usepackage{amsmath}
\usepackage{latexsym}
\usepackage{amsfonts}
\usepackage{color}

\DeclareSymbolFont{rsfscript}{OMS}{rsfs}{m}{n}
\DeclareSymbolFontAlphabet{\mathrsfs}{rsfscript} \DeclareSymbolFont{rsfscript}{OMS}{rsfs}{m}{n} \newcommand{\rf}{\rightarrow}
\newcommand{\sub}{\subseteq}
\newcommand{\men}{\setminus}
\newtheorem{theorem}{Theorem}
\newtheorem{proposition}{Proposition}
\newtheorem{definition}{Definition}
\newtheorem{lemma}{Lemma}
\newtheorem{corollary}{Corollary}
\newtheorem{conj}{Conjecture}

\newtheorem{rem}{Remark}

\title{Union-Closed vs Upward-Closed Families of Finite Sets}
\author{
        Emanuele Rodaro\\
        Centro de Matem\'{a}tica, Faculdade de Ci\^{e}ncias\\
        Universidade do Porto, 4169-007 Porto, Portugal\\
        emanuele.rodaro@fc.up.pt
}
\date{}

\begin{document}
\maketitle

\begin{abstract}
A finite family $\mathrsfs{F}$ of subsets of a finite set $X$ is union-closed whenever $f,g\in\mathrsfs{F}$ implies $f\cup g\in\mathrsfs{F}$. These families are well known because of Frankl's conjecture \cite{Duffus0}. In this paper we developed further the connection between union-closed families and upward-closed families started in \cite{Reimer} using rising operators. With these techniques we are able to obtain tight lower bounds to the average of the length of the elements of $\mathrsfs{F}$ and to prove that the number of joint-irreducible elements of $\mathrsfs{F}$ can not exceed $2{n\choose \lfloor n/2\rfloor}+{n\choose \lfloor n/2\rfloor+1}$ where $|X| = n$.
\end{abstract}

\tableofcontents

\section{Introduction}
Consider a finite set $X=\{a_1,\ldots, a_n\}$ formed by $n\ge 1$ elements. A family of subsets $\mathrsfs{F}$ of the powerset $2^X$ such that for any $f,g\in \mathrsfs{F}$, $f\cup g\in\mathrsfs{F}$ is called union-closed (briefly $\cup$-closed). Without loss of generality we can assume that $X = \bigcup_{f\in\mathrsfs{F}} f$ and for the rest of the paper, when it is not differently stated, $\mathrsfs{F}$ will denote a $\cup$-closed family on $X = \{a_1,\ldots, a_n\}$ with $X = \bigcup_{f\in\mathrsfs{F}} f$. In $1979$, Frankl stated the following conjecture
\begin{conj}\label{cj:Frankl}
  For all union-closed families $\mathrsfs{F}$, there exists an $a\in X$ such that $|\{f\in\mathrsfs{F}: a\in f\}|\ge \frac{|\mathrsfs{F}|}{2}$.
\end{conj}
Although many attempts to solve this simple-sounding conjecture have been made, this remains open and has become known as the \emph{union-closed conjecture} or \emph{Frankl's conjecture}. A simple argument in \cite{Knill} shows that there is an $a\in X$ which is contained in at least $|\mathrsfs{F}|/\log_2(|\mathrsfs{F}|)$ elements of $\mathrsfs{F}$. In \cite{Wojcik}, this bound is improved by a multiplicative constant. The conjecture holds if $|\mathrsfs{F}| < 40$ (see \cite{Lo Faro, Roberts}) or $|X|\le 11$ (see \cite{Markovich, Morris}) or $|\mathrsfs{F}|> \frac{5}{8} 2^{|X|}$ (see \cite{Czedeli1, CzedeliMaroti,CzedeliSchmidt}) or $\mathrsfs{F}$ contains some collection of small sets (see \cite{Markovich, Morris}).
\\
The family $\mathrsfs{F}$ is a semilattice with respect to the union operation, furthermore, since $\mathrsfs{F}$ is finite we can endow $\mathrsfs{F}\cup \{\{\emptyset\}\}$ with a structure of lattice. In this direction it is possible to give another formulation of Frankl's conjecture in the framework of lattice theory. Let $(L,\vee,\wedge)$ be a finite lattice, we denote by $J(L)$ the set of \emph{join-irreducible} elements, i.e., the elements $z\in L$ such that if $g=x\vee y$ then $x=z$ or $y=z$. Denoting by $V_x=\{y\in L:y\le x\}$ the principal filter generated by $x$, Frankl's conjecture is equivalent to the following lattice theoretic conjecture
\begin{conj}
    $$
  \frac{1}{|L|}\min\{|V_x|:x\in J(L)\}\le\frac{1}{2}.
    $$
\end{conj}
This approach has received a significant amount of attention (see \cite{Abe1, Abe2, Abe3, Abe4, Czedeli1, CzedeliMaroti, CzedeliSchmidt, Duffus, Poonen, Reinhold, Stanley}). Although much research has been done on union-closed families, it seems there is no general tool to tackle this problem. In a different direction, Reimer in \cite{Reimer} developed a connection between $\cup$-closed sets and upward-closed sets by using a repeated application of rising operators. From this connection he provided a lower bound on the average of the length of the elements of $\mathrsfs{F}$ showing
$$
\frac{1}{|\mathrsfs{F}|}\sum_{f\in \mathrsfs{F}}|f|\ge \frac{\log_2(|\mathrsfs{F}|)}{2}
$$
The aim of this paper is to develop further the correspondence introduced by Reimer and study more deeply the consequences and some of the results that can be achieved from this point of view. The hope is to give an approach to the study of $\cup$-closed families of sets that can be helpful to give some insight to a possible solution of the Conjecture \ref{cj:Frankl}.\\
The paper is organized as following: in Section \ref{sec:preliminaries} we give some definitions and we fix the notation, in Sections \ref{sec:II} and \ref{sec:permutation dependancy} we extend the results of \cite{Reimer}, in Section \ref{sec:frankl} we use this approach to obtain some lower bounds on the localized average of the length of the elements of $\mathrsfs{F}$. More precisely, given $S\subseteq\mathrsfs{F}$, we provide lower bounds to the quantity
$$
\frac{1}{|\{f\in\mathrsfs{F}: \exists z\in S, z\subseteq f\}|}\sum_{f\in\mathrsfs{F}: \exists z\in S, z\subseteq f}|f|
$$
Finally in Section \ref{sec: bounf join irreducible} we use these techniques to prove that the number of joint-irreducible elements of $\mathrsfs{F}\subseteq 2^X$ is at most $2{n\choose \lfloor n/2\rfloor}+{n\choose \lfloor n/2\rfloor+1}$ where $|X| = n$.

\section{Preliminaries}\label{sec:preliminaries}
For an element $t\in 2^X$ we denote the cardinality of $t$ by $|t|$. Let us fix a subset $S\subseteq 2^X$ (note that the cardinality of $S$ is also denoted by $|S|$). Let $a\in X$, $S$ is partitioned into two subsets $S_a,S_{\overline{a}}$ of the elements of $S$ containing $a$, not containing $a$, respectively. We can see $S$ endowed with the order induced by the relation $\subseteq$ as a poset $(S,\subseteq)$, thus an \emph{antichain} $A\subseteq S$ is a non-empty subset such that any pair of elements of $A$ is incomparable. We denote the set of minimal (maximal) elements of $S$ by $\min(S)$ ($\max(S)$). Both $\min(S), \max(S)$ are clearly antichains. Let us denote the set of all $f^c$ for $f\in S$, where $f^c = X\setminus f$ is the complement set of $f$, by $S^c$. Given two different elements $f,g \in S$ we say that $g$ \emph{covers} $f$, written $f \lessdot g$ if there is no $h\in S$ such that $f\subsetneq h\subsetneq g$.
\\
An \emph{upward-closed family} (also called upset or filter, see \cite{Anderson}) is a subset $\mathcal{F}\sub 2^X$ such that if $f\sub g$ for some $f\in\mathcal{F}, g\in 2^X$, then $g\in\mathcal{F}$, and a \emph{downward-closed family} (also downset or simplicial complex) is defined analogously. Note that an upward-closed family is a $\cup$-closed set.
\\
Let us consider a $\cup$-closed family $\mathrsfs{F}$ of $2^X$. An ideal of $\mathrsfs{F}$ is a subset $I\sub \mathrsfs{F}$ such that $I\cup g\sub I$ for all $g\in \mathrsfs{F}$. Let $z\in 2^X$, the \emph{principal ideal of $z$} denoted by $\mathrsfs{F}[z]$, is the (possibly empty) set of all the elements of $\mathrsfs{F}$ containing $z$. Clearly if $z\in \mathrsfs{F}$, then $\mathrsfs{F}[z]$ is the principal ideal generated by $z$, i.e., $\mathrsfs{F}[z]=z\cup\mathrsfs{F} = \{z\cup f, f\in \mathrsfs{F}\}$. This definition can be extended to any sub-family $S\sub\mathrsfs{F}$, thus the \emph{principal ideal generated by $S$} is the set $\mathrsfs{F}[S]=\cup_{z\in S}\mathrsfs{F}[z]$. It is straightforward to see that $\mathrsfs{F}[\min(\mathrsfs{F})]=\mathrsfs{F}$. In the particular case $\mathrsfs{F} = 2^X$, we use the shorter notation $S^{\uparrow}$ for the set $2^X[S]$.
\\
An element $g\in \mathrsfs{F}$ is called \emph{irreducible} whenever $g = h\cup t$ implies $h = g$ or $t = g$. We denote by $J(\mathrsfs{F})$ the set of irreducible elements of $\mathrsfs{F}$. It is evident that $min(\mathrsfs{F})\subseteq J(\mathrsfs{F})$. This set plays an important role since it is the minimal set of generators of the semilattice $(\mathrsfs{F}, \cup)$. This set is \emph{$\cup$-independent}, in the following sense. Given $S\sub 2^X$, we say that $S$ is $\cup$-independent whenever no element $z\in S$ can be written as a union of elements in $S\setminus\{z\}$.

\section{Union-closed and upward-closed families}\label{sec:II}
In this section we further explore the connection between $\cup$-closed families and upward-closed families. This connection has been already established in \cite{Reimer} using the concept of \emph{rising function}, a well-known operator used also by Frankl in \cite{Frankl}. We briefly recall  such operator. Given an element $a\in X$ and a family (not necessarily $\cup$-closed) of subsets $S\sub 2^X$, the rising function $\varphi_{S,a}:S\rightarrow 2^X$ is the function defined for all $z\in S$ by
$$
\begin{array}{l}
\varphi_{S,a}(z)=\left\{\begin{array}{l}
z\cup\{a\}\;\text{ if } z\cup\{a\}\notin S,\\
z \text{ otherwise};
\end{array}\right.
\end{array}
$$
This is a one-to-one function $\varphi_{S,a}:S\rf 2^X$ and the image $\varphi_{S,a}(S)$ is called the $a$-rising of $S$. In \cite{Reimer}, the author iterates these rising functions in the following way. Let $X=\{a_1,\ldots,a_n\}$ and let $\varphi_0$ be the identity function on $2^X$, and $S_0=S$, then for all $1\le j\le n$ let
$$
S_j=\varphi_{j}(S_{j-1}),\;\;\varphi_{j}=\varphi_{S_{j-1},a_j}\circ\varphi_{j-1}
$$
We call $\varphi_n$ the \emph{rising function with respect to the word} $w=a_1a_2\ldots a_n$ of the family $S$ and we denote it by $\varphi_w$ to underline the dependency of this map from the order used to perform these iterations. We call each $S_i$ the \emph{$i$-section} and for any $z\in S$ the elements $z_{i}=\varphi_{i}(z)$ for $i=0,\ldots, n$ is called the \emph{trajectory} of $z$ through the iterated application of the rising functions.
\\
We immediately note that this definition depends on the order in which the rising functions are iterated. Indeed consider the set $X=\{a,b,c\}$ and the $\cup$-closed family $\mathrsfs{F}=\{\{a\},\{a,b,c\}\}$, it is evident that $\varphi_w(\mathrsfs{F})\neq\varphi_{w'}(\mathrsfs{F})$ when $w=abc, w'=acb$. We denote by $\mathfrak{S}_X$ the permutation group of the set of objects $X = \{a_1,\ldots, a_n\}$, and for a word $w=a_1\ldots a_n$, we use the notation $w\theta = \theta(a_1)\ldots\theta(a_n)$. There is an evident action of $\mathfrak{S}_X$ on the set $\{\varphi_{w\vartheta}(g): g\in\mathrsfs{F}, \vartheta\in \mathfrak{S}_X\}$ given by $\sigma\cdot \varphi_{w\vartheta}(g) = \varphi_{w\vartheta\sigma}$. In Section \ref{sec:permutation dependancy} we characterize the orbits of such action and we explore some consequences.
\\
It is not difficult to see that the rising function $\varphi_w$ is a bijection between the family $S$ and its image $\varphi_w(S)$, moreover from the definition it is easy to verify that, independently from the condition that $S$ is $\cup$-closed, $\varphi_w(S)$ is upward-closed. We have the following lemma:
\begin{lemma}\label{lem:matching property}
  With the notation above, if there are two different elements $z,z'\in S$ such that $\varphi_i(z)=\varphi_i(z')\cup\{a_{i+1}\}$, then $a_{i+1}\in z\men\varphi_w(z')$.
\end{lemma}
\begin{proof}
  If $a_{i+1}\notin z$, then $a_{i+1}\notin\varphi_k(z)$ for all $k\le i$, but this contradicts $\varphi_i(z)=\varphi_i(z')\cup\{a_{i+1}\}$, thus $a_{i+1}\in z$. Since $\varphi_i$ is a bijection and $\varphi_i(z)=\varphi_i(z')\cup\{a_{i+1}\}$ with $z\neq z'$, we get $a_{i+1}\notin\varphi_i(z')$ and so $a_{i+1}\notin\varphi_w(z')$.
\end{proof}
If we add the $\cup$-closed condition, we have the following lemma:
\begin{lemma}\cite{Reimer}\label{lem:the family is upward-closed}
  Let $\mathrsfs{F}$ be a $\cup$-closed family of subsets of $X=\{a_1,\ldots,a_n\}$. For each $0\le i\le n$ the $i$-section $\mathrsfs{F}_i$ is a $\cup$-closed family.
\end{lemma}
The following lemma is a consequence of \cite[Lemma 3.3]{Reimer}, but for the sake of completeness we present here with proof.
\begin{lemma}\label{lem:closure in sections}
  Let $\mathrsfs{F}$ be a $\cup$-closed family of sets and let $f\in\mathrsfs{F}$. Consider the rising function $\varphi_w$ with respect to the word $w=a_1a_2\ldots a_n$ of the family $\mathrsfs{F}$.
  Then if $t$ belongs to the $i$-section $\mathrsfs{F}_i$, for some $1\le i\le n$, then also $t\cup f\in\mathrsfs{F}_i$.
\end{lemma}
\begin{proof}
  We prove it by induction on the index $i$. Suppose that $i=0$, since $\mathrsfs{F}_0=\mathrsfs{F}$ is a $\cup$-closed family, if $t\in\mathrsfs{F}$ then, since $f\in\mathrsfs{F}$, $t\cup f\in\mathrsfs{F}=\mathrsfs{F}_0$. Suppose that the statement of the theorem is true for $i$ and let us prove it for $i+1$. Suppose $t\in\mathrsfs{F}_{i+1}$ and let $\overline{t}=\varphi_{\mathrsfs{F}_{i}, a_{i+1}}^{-1}(t)\in\mathrsfs{F}_{i}$. By the inductive hypothesis $\overline{t}\cup f\in\mathrsfs{F}_{i}$, we consider several cases.
  \\
  \emph{Case 1}. Suppose $a_{i+1}\in \overline{t}$. Thus $a_{i+1}\in \overline{t}\cup f\in\mathrsfs{F}_{i}$, hence $\overline{t}=\varphi_{\mathrsfs{F}_{i}, a_{i+1}}(\overline{t})=t$ and $\varphi_{\mathrsfs{F}_{i}, a_{i+1}}(\overline{t}\cup f)=\overline{t}\cup f=t\cup f$ and so $t\cup f\in\mathrsfs{F}_{i+1}$.
  \\
  \emph{Case 2}. Suppose $a_{i+1}\notin \overline{t}$. We consider two further subcases.
  \begin{itemize}
    \item $a_{i+1}\in t$, hence necessarily by definition of the rising function $\varphi_{\mathrsfs{F}_{i}, a_{i+1}}$, $\overline{t}\cup\{a_{i+1}\}\notin\mathrsfs{F}_{i}$. Therefore $\varphi_{\mathrsfs{F}_{i}, a_{i+1}}(\overline{t})=\overline{t}\cup\{a_{i+1}\}=t$ and, if $\overline{t}\cup f\cup\{a_{i+1}\}\notin \mathrsfs{F}_{i}$ then $\varphi_{\mathrsfs{F}_{i}, a_{i+1}}(\overline{t}\cup f)=\overline{t}\cup f\cup\{a_{i+1}\}=t\cup f\in\mathrsfs{F}_{i+1}$. Otherwise $\overline{t}\cup f\cup\{a_{i+1}\}\in \mathrsfs{F}_{i}$, thus $t\cup f\in\mathrsfs{F}_{i}$, hence $\varphi_{\mathrsfs{F}_{i}, a_{i+1}}(t\cup f)=t\cup f\in\mathrsfs{F}_{i+1}$.
    \item $a_{i+1}\notin t$, hence necessarily $\overline{t}\cup a_{i+1}\in\mathrsfs{F}_{i}$ and $t=\overline{t}$. Thus, if $a_{i+1}\in \overline{t}\cup f$ then $\varphi_{\mathrsfs{F}_{i}, a_{i+1}}(\overline{t}\cup f)=\overline{t}\cup f=t\cup f\in \mathrsfs{F}_{i+1}$. Otherwise $a_{i+1}\notin \overline{t}\cup f$. Since $\overline{t}\cup a_{i+1}\in\mathrsfs{F}_{i}$, then by the inductive hypothesis $\overline{t}\cup a_{i+1}\cup f\in\mathrsfs{F}_{i}$ whence $\varphi_{\mathrsfs{F}_{i}, a_{i+1}}(\overline{t}\cup f)=\overline{t}\cup f=t\cup f\in\mathrsfs{F}_{i+1}$.
  \end{itemize}
\end{proof}
The following theorem establishes an interesting property of the associate upward-closed family $\mathcal{F}=\varphi_w(\mathrsfs{F})$.
\begin{theorem}\label{theo:ideal correspondence}
  For each $f\in \mathrsfs{F}$, $\varphi_w$ is a bijection between the principal ideals
  $$
  \varphi_w:\mathrsfs{F}[f]\rf\mathcal{F}[f]
  $$
\end{theorem}
\begin{proof}
  Since $\varphi_w$ is a one to one function it is sufficient to prove that $
  \varphi_w:\mathrsfs{F}[f]\rf\mathcal{F}[f]$ is also surjective. Since $f\in\mathrsfs{F}$, then $\varphi_w(\mathrsfs{F}[f])\sub\mathcal{F}[f]$, in particular $\mathcal{F}[f]$ is non-empty. Consider an element $\eta\in\mathcal{F}[f]$ and let $\eta^*=\varphi_w^{-1}(\eta)$. We claim that $f\sub \eta^*$ and so $\eta^*\in\mathrsfs{F}[f]$. Suppose, contrary to our claim, that $f\nsubseteq \eta^*$. Let $\eta_0=\eta^*$ and $\eta_{i}=\varphi_{i}(\eta_0)$ for $i = 1,\ldots, n$ be the trajectory of $\eta^*$. Since $f\sub \eta=\eta_n$, there is a minimal index $j\le n$ such that $f\sub\eta_j$ and $j>0$ ($f\nsubseteq \eta_0$). By the minimality of $j$, $f\nsubseteq\eta_{j-1}$. Since $f\sub\eta_j$, we have $a_{j}\notin\eta_{j-1}$ and so $a_{j}\in f$. By Lemma \ref{lem:closure in sections}, since $\eta_{j-1}\in\mathrsfs{F}_{j-1}$, we get also $\eta_{j-1}\cup f\in\mathrsfs{F}_{j-1}$. Therefore, since $a_{j}\in f$, $f\sub\eta_{j-1}\cup a_j$ and so $\eta_{j-1}\cup a_j=\eta_{j-1}\cup f\in\mathrsfs{F}_{j-1}$, hence $f\nsubseteq\eta_j=\eta_{j-1}$, a contradiction.
\end{proof}
\noindent We have the following corollary:
\begin{corollary}\label{cor:ideal correspondence}
  For each $S\sub\mathrsfs{F}$, $\varphi_w$ is a bijection between the principal ideals
  $$
  \varphi_w:\mathrsfs{F}[S]\rf\mathcal{F}[S]
  $$
  Moreover the inverse of $\varphi_w:\mathrsfs{F}\rightarrow\mathcal{F}$ is given by
  $$
  \varphi_w^{-1}(\eta)=\bigcup_{\{f\in\mathrsfs{F}:f\sub \eta\}}f
  $$
\end{corollary}
\begin{proof}
  Since $\varphi_w$ is injective, it is sufficient to prove that it is also surjective. Thus, consider $\eta\in\mathcal{F}[S]$, then there is an $f\in S$ such that $f\sub \eta$, and so $\eta\in\mathcal{F}[f]$. Therefore, by Theorem \ref{theo:ideal correspondence} $\varphi_w^{-1}(\eta)\in\mathrsfs{F}[f]\sub\mathrsfs{F}[S]$.
  \\
  Let us prove the last statement, so consider an element $\eta\in\mathcal{F}$. By the previous statement $\mathcal{F}=\mathcal{F}[\min(\mathrsfs{F})]$. Hence the set $\{f\in\mathrsfs{F}:f\sub \eta\}$ is non-empty and so, since $\mathrsfs{F}$ is union-closed:
  $$
  \bigcup_{\{f\in\mathrsfs{F}:f\sub \eta\}}f=\eta^*\in\mathrsfs{F}
  $$
  By Theorem \ref{theo:ideal correspondence} and $\eta^*\subseteq\eta$, we get $\varphi_w^{-1}(\eta)\sub \eta$ and $\eta^*\sub\varphi_w^{-1}(\eta)$. Therefore we get $\varphi_w^{-1}(\eta)\sub\eta^*\sub\varphi_w^{-1}(\eta)$, i.e. $\eta^*=\varphi_w^{-1}(\eta)$.
\end{proof}
We give a lemma useful in the sequel.
\begin{lemma}\label{lem:embedding of the elements no having a}
    The map $\psi(z)=z\cup\{a\}$ is an embedding
  $$
  \psi:\mathcal{F}_{\overline{a}}\hookrightarrow\varphi_w(\mathrsfs{F}_a)
  $$
\end{lemma}
\begin{proof}
  Since $\psi:\mathcal{F}_{\overline{a}}\rf 2^X$ is already injective, it is sufficient to prove $\psi(\mathcal{F}_{\overline{a}})\sub \varphi_w(\mathrsfs{F}_a)$. Suppose, contrary to our claim, that there is $\eta\in\mathcal{F}_{\overline{a}}$ such that $z=\varphi_w^{-1}(\eta\cup\{a\})$ is not in $\mathrsfs{F}_a$. Since $a\notin z$ and $z\sub \eta\cup\{a\}$, then $z\sub \eta$. Let $z'=\varphi_w^{-1}(\eta)$, since $z\sub \eta$, then by Theorem \ref{theo:ideal correspondence} we get $z\sub z'$. On the other side, since $z'\sub\eta\sub \eta\cup\{a\}$ then by Theorem \ref{theo:ideal correspondence} $z'\sub \varphi_w^{-1}(\eta\cup\{a\})=z$, whence $z=z'$ which implies $\eta=\eta\cup\{a\}$, a contradiction.
\end{proof}
We say that $g\in\mathrsfs{F}$ is \emph{fixed by $\varphi_w$} whenever $\varphi_w(g)=g$ holds. The following proposition characterized the elements of $\mathrsfs{F}$ with this property.
\begin{proposition}\label{prop:elements fixed by the rising function}
  $\varphi_w(g)=g$ if and only if $g\cup a\in\mathrsfs{F}$ for all $a\in X$. Moreover if $S\sub\mathrsfs{F}$ then $\mathcal{F}\cap S$ is the set of elements of $S$ fixed by $\varphi_w$.
\end{proposition}
\begin{proof}
  Using the definition of $\varphi_w$ and Lemma \ref{lem:closure in sections} it is straightforward to check that if $g\cup a\in\mathrsfs{F}$ for all $a\in X$ then $\varphi_w(g)=g$. Conversely, suppose that $\varphi_w(g)=g$ and let us prove that $g\cup a_i\in \mathrsfs{F}$. Since $\varphi_w(g)=g$, then if $g_i$ is the trajectory of $g$ in the rising process, then $g_i=g$ for all $i=1,\ldots, n$. In particular $g\cup{a_{i+1}}\in\mathrsfs{F}_i$ by definition of the rising function $\varphi_{\mathrsfs{F}_i,a_{i+1}}$. By Lemma \ref{lem:matching property} $a_{i+1}\in \varphi^{-1}_{i}(g\cup{a_{i+1}})\setminus g$ and by Corollary \ref{cor:ideal correspondence} $g\subseteq\varphi^{-1}_{i}(g\cup{a_{i+1}})$, whence $g\cup a_{i+1}\subseteq \varphi^{-1}_{i}(g\cup{a_{i+1}})\subseteq g\cup{a_{i+1}}$, i.e. $g\cup\{a_{i+1}\}=\varphi^{-1}_{i}(g\cup{a_{i+1}})\in\mathrsfs{F}$. The proof of the last statement of the lemma is also a consequence of Corollary \ref{cor:ideal correspondence} and it is left to the reader.
\end{proof}
The last proposition shows that all the upward-closed families of sets are leaved unchanged by the rising operator $\varphi_w$.
\\
We remind that if $A\sub B$ are two subsets of $X$ then the \emph{interval} $[A,B]$ is defined by $\{D\sub X:A\sub D\sub B\}$. In \cite[Lemma 1.3 (ii)]{Reimer} the author shows that if $g\neq f$ are two distinct elements of $\mathrsfs{F}$ then $[g,\varphi_w(g)]\cap [f,\varphi_w(f)]=\emptyset$. This facts is independent from the order with which we rise the set, indeed we have the following proposition.
\begin{proposition}\label{prop:interval property}
   Let $f,g\in\mathrsfs{F}$ and $\sigma,\theta\in\mathfrak{S}_X$. Then $f\neq g$ if and only if $[f,\varphi_{w\theta}(f)]\cap [g,\varphi_{w\sigma}(g)]=\emptyset$.
\end{proposition}
\begin{proof}
  Suppose that $z\in[f,\varphi_{w\theta}(f)]\cap [g,\varphi_{w\sigma}(g)]\neq\emptyset$. By Corollary \ref{cor:ideal correspondence} and $f\sub c\sub\varphi_{w\sigma}(g)$ we get $g=\bigcup_{\{h\in\mathrsfs{F}:h\sub \varphi_{w\sigma}(g)\}}h\supseteq f$. Changing $g$ with $f$ we obtain the other inclusion $g\sub f$, whence $g=f$. The other side of the implication is trivial.
\end{proof}

\section{The invariant upward-closed family associated to a union-closed family}\label{sec:permutation dependancy}
In this section we introduce an upward-closed family associated to $\mathrsfs{F}$ which do not depend on a parameter like the case obtained using the rising functions in Section \ref{sec:II}. From Theorem \ref{theo:ideal correspondence} we have that $\varphi_{w}(\mathrsfs{F})$ is an upward-closed family, moreover since the union of upward-closed families is still an upward-closed family, we can associate to $\mathrsfs{F}$ the upward-closed family
\begin{equation}\label{eq: upward-closed family}
\textbf{U}(\mathrsfs{F})=\bigcup_{\vartheta\in \mathfrak{S}_X}\varphi_{w\vartheta}(\mathrsfs{F})
\end{equation}
where $w=a_1\ldots a_n$. We call $\textbf{U}(\mathrsfs{F})$ the \emph{invariant upward-closed family} associated to $\mathrsfs{F}$. We have already noted in Section \ref{sec:II} that there is an action of $\mathfrak{S}_X$ on this set given by $\beta\cdot\varphi_{w\vartheta}(g)=\varphi_{w\vartheta\beta}(g)$. So it seems natural to characterize the orbits $
\mathfrak{S}_X\cdot \varphi_{w}(g) = \{\varphi_{w\vartheta}(g), \vartheta\in \mathfrak{S}_X\}$.
Before giving this characterization we need first some definitions. The rising function $\varphi_w$ depends on the parameter $w$, however by Corollary \ref{cor:ideal correspondence} the inverse of $\varphi_w$ does not. Moreover, by the same Corollary, $\varphi_w(\mathrsfs{F})\subseteq\min(\mathrsfs{F})^{\uparrow}$ and so $\textbf{U}(\mathrsfs{F})\subseteq\min(\mathrsfs{F})^{\uparrow}$. For this reason it is important to extend the map $\varphi_w^{-1}$ to an operator
$$
\circ^*:min(\mathrsfs{F})^{\uparrow}\rightarrow \mathrsfs{F}
$$
which associates to each element $z\in min(\mathrsfs{F})^{\uparrow}$ the element
$$
  z^* = \bigcup_{\{h\in\mathrsfs{F}, h\subseteq z\}} h
$$
Using the fact that $\mathrsfs{F}$ is $\cup$-closed and the domain is $min(\mathrsfs{F})^{\uparrow}$, it is immediate to see that this operator is well defined. Moreover $\circ^*$ preserves inclusion, i.e. if $z\subseteq y$ then $z^*\subseteq y^*$ and it is clearly surjective, thus we can define the \emph{fiber} of each $g\in\mathrsfs{F}$ as
$$
Fib(g) = \{h\in min(\mathrsfs{F})^{\uparrow}: h^* = g\}.
$$
The following proposition characterizes the union-closed families in term of the operator $\circ^*$.
\begin{proposition}
  Let $\mathrsfs{H}$ be a family of subsets of $X$ and consider the operator
$$
\circ^*:min(\mathrsfs{H})^{\uparrow}\rightarrow 2^X
$$
defined by sending each $z\in min(\mathrsfs{H})^{\uparrow}$ into $z^* = \bigcup_{\{h\in\mathrsfs{F}, h\subseteq z\}} h$. Then $\mathrsfs{H}$ is a $\cup$-closed family if and only if the image of $\circ^*$ is contained in $\mathrsfs{H}$.
\end{proposition}
\begin{proof}
As we have already noticed before if $\mathrsfs{H}$ is $\cup$-closed then $\circ^*$ is well defined map $\circ^*:min(\mathrsfs{H})^{\uparrow}\rightarrow \mathrsfs{H}$. Conversely, let $g,h\in \mathrsfs{H}$ and let $\mathrsfs{H}'\subseteq \mathrsfs{H}$ be the image of $\mathrsfs{H}$ by means of $\circ^*$. The element $g\cup h\in min(\mathrsfs{H})^{\uparrow}$ and so $g\cup h\in Fib(t)$ for some $t\in \mathrsfs{H}'$. Since $Fib(t)$ is formed by the elements $z$ such that $z^* = t$ and $t\in\mathrsfs{H}'$ we have that $t\subseteq z$ for all $z\in Fib(t)$, in particular $t\subseteq g\cup h$. On the other hand $g,h\subseteq g\cup h$ and so $g,h\subseteq (g\cup h)^* = t$,
whence $g\cup h\subseteq t$ and so $g\cup h = t\in\mathrsfs{H}'\subseteq \mathrsfs{H}$.
\end{proof}
Given a word $u = w\theta = a_{i_1}\ldots a_{i_n}$ for some $\theta\in \mathfrak{S}_X$ and a subset $\gamma\subseteq X=\{a_{1},\ldots, a_{n}\}$ we say that $\gamma$ is \emph{contained in a prefix} of $u$ (or $u$ has a prefix containing $\gamma$) whenever either $\gamma$ is empty or there is a prefix $u'=a_{i_1}\ldots a_{i_l}$ of $u$ for some $l$ with $n \ge l \ge 1$ with $\gamma = \{a_{i_1},\ldots, a_{i_l}\}$.
\begin{lemma}\label{lem:word generatin a maximal fiber}
  Let $\mathrsfs{F}$ be a $\cup$-closed family of sets of $X=\{a_1,\ldots,a_n\}$, let $g\in \mathrsfs{F}$ and $\eta\in \max(Fib(g))$. Then for any word $u = a_{i_1}\ldots a_{i_n}$ having a prefix containing $\eta\setminus\eta^*$ we have $\varphi_{u}(g) = \eta$.
\end{lemma}
\begin{proof}
Suppose $\eta\setminus\eta^* \neq \emptyset$ (the empty case can be treated analogously) and let $u = a_{i_1}\ldots a_{i_n}$ be a word with the property of the statement and so there is some $l$ with $n \ge l \ge 1$ such that $\eta\setminus\eta^*=\{a_{i_1},\ldots, a_{i_l}\}$. Let $\eta_0=g$ and $\eta_{j}=\varphi_{j}(g)$ for $j=1,\ldots, n$ be the trajectory of $g$ trough the iterated application of the rising functions with respect to $u$ and let $\mathrsfs{F}_j$ be the associated sections. Suppose that there is an integer $s$ with $0\le s < l$ such that $\eta_s\cup a_{i_{s+1}}\in \mathrsfs{F}_s$ and let us suppose without loss of generality that such $s$ is minimum between the integers with this property. Since $\eta_s\cup a_{i_{s+1}}\in \mathrsfs{F}_s$ there is an element $f\in \mathrsfs{F}$ with $f\neq g$ such that $\eta_s\cup a_{i_{s+1}} = \varphi_s(f)$. Thus, since $g\subseteq \eta_s$ we get $g\subseteq \eta_s\cup a_{i_{s+1}} = \varphi_s(f)\subseteq\varphi_w(f)$ and so by Corollary \ref{cor:ideal correspondence} we have $g\subsetneq f$. Since $a_{i_1}\ldots a_{i_l}$ is a prefix of $u$ and $\{a_{i_1},\ldots, a_{i_l}\} = \eta\setminus\eta^*$, $g= \eta_0\subseteq\eta$ then $\eta_s\subseteq\eta$, moreover since $s<l$ then $a_{i_{s+1}}\in \eta$, hence $\eta_s\cup a_{i_{s+1}}\subseteq\eta$. Therefore we have the contradiction:
$$
g = \eta^* \supseteq (\eta_s\cup a_{i_{s+1}})^* = f \supsetneq g
$$
since by Corollary \ref{cor:ideal correspondence} $f = \varphi_{u}(f)^*\supseteq(\eta_s\cup a_{i_{s+1}})^*\supseteq f$.
Hence we can suppose that for all $0\le s<l$ we have $\eta_s\cup a_{i_{s+1}}\notin \mathrsfs{F}_s$ and so we have $\eta_l = \eta$. Thus $\eta\subseteq\varphi_{u}(g)$. Let us prove that actually $\eta = \varphi_{u}(g)$. Suppose on the contrary that $\eta\subsetneq\varphi_{u}(g)$, since by Corollary \ref{cor:ideal correspondence} $g = (\varphi_{u}(g))^*$ then $\varphi_{u}(g)\in Fib(g)$, however $\eta\subsetneq\varphi_{w'}(g)$ contradicts the maximality of $\eta$, hence $\eta = \varphi_{u}(g)$.
\end{proof}
The following theorem characterizes the orbits of $\textbf{U}(\mathrsfs{F})$.
\begin{theorem}\label{theo: characterization orbits}
  Let $\mathrsfs{F}$ be a $\cup$-closed family of sets of $X=\{a_1,a_2,\ldots,a_n\}$ and let $w=a_1a_2\ldots a_n$. Let $g\in \mathrsfs{F}$ then:
$$
\mathfrak{S}_X\cdot \varphi_{w}(g) = \max(Fib(g))
$$
\end{theorem}
\begin{proof}
  The inclusion $\max(Fib(g)) \subseteq \mathfrak{S}_X\cdot \varphi_{w}(g)$ is a consequence of Lemma \ref{lem:word generatin a maximal fiber}.
  On the other hand, let $\varphi_{w'}(g)\in \mathfrak{S}_X\cdot \varphi_{w}(g)$ for some $w'=w\theta$, $\theta\in \mathfrak{S}_X$. By Corollary \ref{cor:ideal correspondence} $(\varphi_{w'}(g))^* = g$, thus we have $\varphi_{w'}(g)\in Fib(g)$. Suppose, contrary to the statement of the lemma, that $\varphi_{w'}(g)$ is not maximal in $Fib(g)$ and so let $\eta'\in Fib(g)$ such that $\varphi_{w'}(g)\subsetneq\eta'$. Since $\varphi_{w'}(\mathrsfs{F})$ is an upward-closed set and $\varphi_{w'}(g)\in \varphi_{w'}(\mathrsfs{F})$ with $\varphi_{w'}(g)\subsetneq\eta'$, then we get $\eta'\in\varphi_{w'}(\mathrsfs{F})$. However, by Corollary \ref{cor:ideal correspondence} we have the contradiction $g\subsetneq (\eta')^* = g$. Hence $\varphi_{w'}(g)\in \max(Fib(g))$ and so $\mathfrak{S}_X\cdot \varphi_{w}(g) \subseteq \max(Fib(g))$.
\end{proof}
Note that Theorem \ref{theo: characterization orbits}, together with the fact that $Fib(g)\cap Fib(f)=\emptyset$ iff $g\neq f$, implies Proposition \ref{prop:interval property}, in particular we have
$$
Fib(g) = \bigcup_{\vartheta\in \mathfrak{S}_X} [g, \varphi_{w\vartheta}(g)]
$$
Using the invariant upward-closed family $\textbf{U}(\mathrsfs{F})$ we can give tights upper and lower bounds to $|\mathrsfs{F}|$ depending on $rk(\mathrsfs{F})=\min\{|\eta|:\eta\in\min(\textbf{U}(\mathrsfs{F}))\}$. We have the following proposition:
\begin{proposition}
  $$
  2^{n-rk(\mathrsfs{F})}\le|\mathrsfs{F}|\le \sum_{i\ge rk(\mathrsfs{F})}{n\choose i}
  $$
  and these bounds are tights.
\end{proposition}
\begin{proof}
   Let $z\in\min(\textbf{U}(\mathrsfs{F}))$ with $|z|=rk(\mathrsfs{F})$, then by Theorem \ref{theo: characterization orbits} $z=\varphi_{w\theta}(g)$ for some $\theta\in \mathfrak{S}_X$, thus $z^{\uparrow}\subseteq \varphi_{w\theta}(\mathrsfs{F})$. Thus $|\mathrsfs{F}|=|\varphi_{w\theta}(\mathrsfs{F})|\ge |z^{\uparrow}|=2^{n-rk(\mathrsfs{F})}$. This bound is attained considering the $\cup$-closed family $\{\overline{z}\}^{\uparrow}$. By Proposition \ref{prop:elements fixed by the rising function} $\varphi_{w\theta}(\mathrsfs{F})=\{\overline{z}\}^{\uparrow}$ for all $\theta\in \mathfrak{S}_X$, thus $\textbf{U}(\mathrsfs{F})=\{\overline{z}\}^{\uparrow}$ and so $rk(\mathrsfs{F})=|\overline{z}|$. The upper bound is obtained in a similar way and its proof is left to the reader.
\end{proof}
Let $x\in\textbf{U}(\mathrsfs{F})$, $\mathfrak{S}_x$ denotes the stabilizer subgroup of $x$ and as usual by $\textbf{U}(\mathrsfs{F})^{\vartheta}$ the set of elements of $\textbf{U}(\mathrsfs{F})$ fixed by an element $\vartheta\in \mathfrak{S}_X$. As a consequence of Theorem \ref{theo: characterization orbits} and Burnside's Lemma we have the following corollary:
\begin{corollary}
    $$
        |\mathrsfs{F}| = \frac{1}{n!}\sum_{x\in \textbf{U}(\mathrsfs{F})} |\mathfrak{S}_x|
    $$
    In particular we have the following inequality:
    $$
    \frac{1}{|\mathrsfs{F}|}\sum_{f\in\mathrsfs{F}} \sum_{x\in \max(Fib(f))}\frac{1}{{n\choose |x\setminus x^*|}}\le 1
    $$
\end{corollary}
\begin{proof}
  Using Burnside's Lemma
  $$
  |\textbf{U}(\mathrsfs{F})/\mathfrak{S}_X| = \frac{1}{|\mathfrak{S}_X|}\sum_{\vartheta\in\mathfrak{S}_X}|\textbf{U}(\mathrsfs{F})^{\vartheta}|= \frac{1}{n!}\sum_{x\in \textbf{U}(\mathrsfs{F})}|\mathfrak{S}_x|
  $$
  by Theorem \ref{theo: characterization orbits} the set of orbits $\textbf{U}(\mathrsfs{F})/\mathfrak{S}_X$ is in one to one correspondence with $\mathrsfs{F}$, thus $|\textbf{U}(\mathrsfs{F})/\mathfrak{S}_X| = |\mathrsfs{F}|$ and so the equality of the corollary is proved. To prove the inequality we give a lower bound to $|\mathfrak{S}_{x}|$ for $x\in \max(Fib(f))$. By Lemma \ref{lem:word generatin a maximal fiber} we have that for any word $u = a_{i_1}\ldots a_{i_n}$ having a prefix containing $x\setminus x^*$, $\varphi_{u}(g) = x$. There are $|x\setminus x^*|! (n-|x\setminus x^*|)!$ such words and so $|\mathfrak{S}_{x}|\ge |x\setminus x^*|! (n-| x\setminus x^*|)!$. Therefore by Theorem \ref{theo: characterization orbits} we have
  \begin{eqnarray*}
    1 &=& \frac{1}{|\mathrsfs{F}| n!}\sum_{x\in \textbf{U}(\mathrsfs{F})} |\mathfrak{S}_{x}|\ge \frac{1}{|\mathrsfs{F}|}\sum_{x\in \textbf{U}(\mathrsfs{F})} \frac{(|x \setminus x^*|)! (n-|x\setminus x^*|)!}{n!} \\
     &=& \frac{1}{|\mathrsfs{F}|}\sum_{f\in\mathrsfs{F}}\sum_{x\in \max(Fib(f))} \frac{1}{{n\choose | x\setminus x^*|}}
  \end{eqnarray*}
\end{proof}
The following lemma characterizes the elements not containing an $a\in X$ for which in the rising process, for some order of rising, the elements will also not contain $a$.
\begin{lemma}\label{lem: covering property}
  Let $g\in \mathrsfs{F}_{\overline{a}}$ and $\eta\in \max(Fib(g))$. Then $a\notin\eta$ if and only if there is $h\in \mathrsfs{F}_a$ such that $h\subseteq \eta\cup \{a\}$ and $g\lessdot h$.
\end{lemma}
\begin{proof}
  Suppose that $a\notin\eta$ and let $h' = (\eta\cup \{a\})^*$. Since $\eta\in Fib(g)$, then $g\subseteq \eta$ and so $g\subseteq h'$. We claim $(h'\setminus \{a\})^* = g$. By Lemma \ref{lem:embedding of the elements no having a} and Theorem \ref{theo: characterization orbits} we get $h'\in\mathrsfs{F}_a$, and by definition of the operator $\circ^*$, $h' \subseteq \eta\cup \{a\}$. Since $g\subseteq h'$ and $g\in\mathrsfs{F}_{\overline{a}}$, then $g\subseteq (h'\setminus \{a\})^*\subseteq \eta^* = g$ and so the claim $(h'\setminus \{a\})^* = g$. Reasoning by contradiction suppose that there is a $g'\in \mathrsfs{F}_{\overline{a}}$ such that $g\lessdot g'\subseteq h'$. Thus $g'\subseteq (h'\setminus \{a\})$ and so we get the contradiction $g\lessdot g'\subseteq (h'\setminus \{a\})^* = g$, whence there is an $h\in \mathrsfs{F}_{a}$ such that $g\lessdot h\subseteq h'$.
  \\
  On the other side, suppose, contrary to the statement of the lemma, that $a\in\eta$. Thus $h\subseteq \eta$, hence we have $h\subseteq \eta^* = g$. However $h\in\mathrsfs{F}_a$ and $g\in \mathrsfs{F}_{\overline{a}}$, a contradiction.
\end{proof}
The following lemma characterizes the elements of $\mathrsfs{F}_{\overline{a}}$ that have at least one maximal element in their fiber that do not contain the element $a$.
\begin{lemma}\label{lem: characterization covering}
  Let $a\in X$ and let $g\in \mathrsfs{F}_{\overline{a}}$. Then there is an $\eta\in \max(Fib(g))$ with $a\notin\eta$ if and only if there is an $h\in\mathrsfs{F}_{a}$ such that $g\lessdot h$.
\end{lemma}
\begin{proof}
  Suppose that there is an $\eta\in\max(Fib(g))$ with $a\notin\eta$. By Lemma \ref{lem: covering property} there is an $h\in\mathrsfs{F}_{a}$ such that $g\lessdot h$.
We prove the other side of the equivalence using an argument similar to the one in Lemma \ref{lem: covering property}. Indeed consider the permutation $(a_{i_1},\ldots, a_{i_n})$ of $X$ with $g = \{a_{i_1},\ldots, a_{i_k}\}$, $h = \{a_{i_1},\ldots, a_{i_l}\}$, $a_{i_l} = a$ for some $n\ge l\ge k$. Consider the word $w' = a_{i_1},\ldots, a_{i_n}$, put $\eta_0=g$ and $\eta_{j}=\varphi_{j}(\eta_0)$ for $j=1,\ldots, n$ be the trajectory of $g$ trough the iterated application of the rising functions with respect to $w'$ and let $\mathrsfs{F}_j$ be the associated sections. We claim that $\varphi_{l-1}(g)= g\cup \{a_{i_{k+1}},\ldots,a_{i_{l-1}}\}$. Clearly $\eta_{k} = g$ and suppose, contrary to our claim, that there is an integer $s$ with $k \le  s < l-1$ such that $\eta_s\cup a_{i_{s+1}}\in \mathrsfs{F}_s$ and let us suppose that $s$ is the minimum between the integers with this property. Since $\eta_s\cup a_{i_{s+1}}\in \mathrsfs{F}_s$ there is an element $g'\in \mathrsfs{F}$ with $g'\neq g$ such that $\eta_s\cup a_{i_{s+1}} = \varphi_s(g')$. Thus, since $g\subseteq \eta_s$ we get $g\subseteq \eta_s\cup a_{i_{s+1}} = \varphi_s(g)\subseteq\varphi_{w'}(g')$ and so by Corollary \ref{cor:ideal correspondence} we have $g\subsetneq g'$. Since $a_{i_1}\ldots a_{i_l}$ is a prefix of $w'$, $h = \{a_{i_1},\ldots, a_{i_l}\}$, $\eta_0\subseteq h$, $s<l-1$ and $a_{i_l} = a$ then $\eta_s\cup a_{i_{s+1}}\subseteq h\setminus \{a\}$. Since $g\lessdot h$ and $g\in \mathrsfs{F}_{\overline{a}}$ it is straightforward to check that $g = (h\setminus \{a\})^*$ and so we have the contradiction:
$$
g = (h\setminus \{a\})^* \supseteq (\eta_s\cup a_{i_{s+1}})^* = g' \supsetneq g
$$
since by Corollary \ref{cor:ideal correspondence} we have $g' = \varphi_{w'}(g')^*\supseteq(\eta_s\cup a_{i_{s+1}})^*\supseteq g'$. Therefore $\eta_s\cup a_{i_{s+1}}\notin \mathrsfs{F}_s$ for all $k\le s < l-1$ and so $\varphi_{l-1}(g)= g\cup \{a_{i_{k+1}},\ldots,a_{i_{l-1}}\}$. Since $h = \varphi_{l-1}(h)\in \mathrsfs{F}_{l-1}$ and $\varphi_{l-1}(g) \cup\{a_{i_l}\} = h$ we have $\varphi_{l-1}(g) \cup\{a_{i_l}\} \in \mathrsfs{F}_{l-1}$ hence $\varphi_{l}(g) = \varphi_{l-1}(g) = h\setminus{a}$ ($a = a_{i_l}$) and so $a\notin\eta_m$ for all $m\ge l$. In particular $a\notin \varphi_{w'}(g)$, whence by Theorem \ref{theo: characterization orbits} $\varphi_{w'}(g)\in\max(Fib(g))$ is the element $\eta$ satisfying the condition of the lemma.
\end{proof}
In view of Lemma \ref{lem: characterization covering} we say that $g\in\mathrsfs{F}_{\overline{a}}$ is \emph{covered in} $a$ if there is an $h\in\mathrsfs{F}_{a}$ such that $g\lessdot h$. In this case we say that $h$ covers $g$ in $a$. The following proposition gives an equivalent formulation of this definition.
\begin{proposition}\label{prop: covered in a}
  $g\in\mathrsfs{F}_{\overline{a}}$ is \emph{covered in} $a$ iff there is an $h\in\mathrsfs{F}_{a}$ such that $(h\setminus\{a\})^* = g$.
\end{proposition}
\begin{proof}
  Suppose that $h\in\mathrsfs{F}_{a}$ such that $g\lessdot h$, then it is straightforward to see that $(h\setminus\{a\})^* = g$. Conversely suppose that there is an $h\in\mathrsfs{F}_{a}$ such that $(h\setminus\{a\})^* = g$. Arguing by contradiction suppose that $g$ is not covered in $a$ and so for any $t\in \mathrsfs{F}_{a}$ there is a $g'\in \mathrsfs{F}_{\overline{a}}$ such that $g\subsetneq g'\subsetneq t$. In particular this occurs for $h$, hence there is a $g'\in \mathrsfs{F}_{\overline{a}}$  with $g\subsetneq g'\subsetneq h$. Thus we have the contradiction $g'\subseteq (h\setminus\{a\})^* = g \subsetneq g'$.
\end{proof}
From this proposition we have that the set
$$
Cov_a(g) = \{h\in\mathrsfs{F}_{a}:(h\setminus\{a\})^* = g \}
$$
is non-empty iff $g$ is covered in $a$.

\section{Some results around Frankl's conjecture}\label{sec:frankl}
The connection between upward-closed families and $\cup$-closed families that we have established in the previous two sections can be useful to try to tackle Frankl's conjecture. The aim of this section is to introduce some subsets which are related to this conjecture. In particular in the first part we fix a word $w$ and we introduce these sets using the rising function $\varphi_w$, in the second part we draw some consequences of this approach giving some lower bounds on the quantity $\frac{1}{|\mathrsfs{F}[S]|}\sum_{f\in\mathrsfs{F}[S]}|f|$, for any $S\subseteq \mathrsfs{F}$, and in the last part we consider the invariant case.

\subsection{Some useful subsets}\label{subsec: useful subsets}
\begin{definition}\label{def:P,S}
  Let $\mathrsfs{H}$ be a family of sets of $X=\{a_1,\ldots,a_n\}$ and let $a\in X$. We denote by $S(\mathrsfs{H},a)$ the set of all the elements $z\in\mathrsfs{H}$ such that $z\cup\{a\}\notin\mathrsfs{H}$. Dually we put $P(\mathrsfs{H},a)$ as the set of all the elements $z\in\mathrsfs{H}$ such that $z\men\{a\}\notin\mathrsfs{H}$.
\end{definition}
Note that $P(\mathrsfs{H},a)$ is non-empty since $\min\{\mathrsfs{H}\}_a \subseteq P(\mathrsfs{H},a)$. We have the following proposition.
\begin{proposition}\label{prop:another caracherization without filter}
  Let $\mathrsfs{H}$ be a family of sets of $X$, then for any $a\in X$:
  $$
  |\mathrsfs{H}_a|-|\mathrsfs{H}_{\overline{a}}|=|P(\mathrsfs{H},a)|-|S(\mathrsfs{H},a)|
  $$
  Moreover if $\mathrsfs{H}$ is $\cup$-closed, then Frankl's conjecture holds for $\mathrsfs{H}$ if and only if there is some $a\in X$ such that
  $$
  |P(\mathrsfs{H},a)|\ge |S(\mathrsfs{H},a)|
  $$
\end{proposition}
\begin{proof}
  It is straightforward to check that the function $\psi$ from the set $\{f\in\mathrsfs{H}_a:f\men\{a\}\in\mathrsfs{H}\}$ onto the set $\{f\in\mathrsfs{H}_{\overline{a}}:f\cup\{a\}\in\mathrsfs{H}\}$ defined by $\psi(z)=z\men\{a\}$ is a bijection. Furthermore $\{f\in\mathrsfs{H}_a:f\men\{a\}\in\mathrsfs{H}\}$ is in bijection with the set $\{f\in\mathrsfs{H}^c_{\overline{a}}:f\cup\{a\}\in\mathrsfs{H}^c\}$ and so $|\{f\in\mathrsfs{H}^c_{\overline{a}}:f\cup\{a\}\in\mathrsfs{H}^c\}|=
  |\{f\in\mathrsfs{H}_{\overline{a}}:f\cup\{a\}\in\mathrsfs{H}\}|$, whence
  \begin{align}
  \nonumber &|\mathrsfs{H}\men\{f\in\mathrsfs{H}_{\overline{a}}:f\cup\{a\}\in\mathrsfs{H}\}|=
  |\mathrsfs{H}|-|\{f\in\mathrsfs{H}_{\overline{a}}:f\cup\{a\}\in\mathrsfs{H}\}|=\\
  \nonumber &=|\mathrsfs{H}^c|-|\{f\in\mathrsfs{H}^c_{\overline{a}}:f\cup\{a\}\in\mathrsfs{H}^c\}|=
  |\mathrsfs{H}^c\men\{f\in\mathrsfs{H}^c_{\overline{a}}:f\cup\{a\}\in\mathrsfs{H}^c\}|
  \end{align}
  Hence from $|\mathrsfs{H}\men\{f\in\mathrsfs{H}_{\overline{a}}:f\cup\{a\}\in\mathrsfs{H}\}|=
  |\mathrsfs{H}^c\men\{f\in\mathrsfs{H}^c_{\overline{a}}:f\cup\{a\}\in\mathrsfs{H}^c\}|$ we get the equality
  $$
  |\mathrsfs{H}_a|+|\{f\in\mathrsfs{H}_{\overline{a}}:f\cup\{a\}\notin\mathrsfs{H}\}|=
  |\mathrsfs{H}^c_a|+|\{f\in\mathrsfs{H}^c_{\overline{a}}:f\cup\{a\}\notin\mathrsfs{H}^c\}|
  $$
  and so the statement follows from $|\mathrsfs{H}^c_a|=|\mathrsfs{H}_{\overline{a}}|$, $|\{f\in\mathrsfs{H}^c_{\overline{a}}:f\cup\{a\}\notin\mathrsfs{H}^c\}|=|
  \{f\in\mathrsfs{H}_{a}:f\men\{a\}\notin\mathrsfs{H}\}|=|P(\mathrsfs{H},a)|$ and $\{f\in\mathrsfs{H}_{\overline{a}}:f\cup\{a\}\notin\mathrsfs{H}\}=S(\mathrsfs{H},a)$.
  The last claim of the proposition is a consequence of $2|\mathrsfs{H}_a|-|\mathrsfs{H}|=|\mathrsfs{H}_a|-|\mathrsfs{H}_{\overline{a}}|$.
\end{proof}
Therefore the study of the sets $S(\mathrsfs{H},a)$ and $P(\mathrsfs{H},a)$ seems important in a possible proof of the Frankl's conjecture. Let us fix a $\cup$-closed family $\mathrsfs{F}$ on $X$, let $\mathcal{F}=\varphi_w(\mathrsfs{F})$ be the associated upward-closed family for some fixed word $w$. We introduce now two analogous sets which are important to give a lower bound to the quantity $\frac{1}{|\mathrsfs{F}[S]|}\sum_{f\in\mathrsfs{F}[S]}|f|$, for any $S\subseteq \mathrsfs{F}$ and which are somehow related to $S(\mathrsfs{H},a)$ and $P(\mathrsfs{H},a)$.
\begin{definition}\label{def:spurious and pure elements}
  Let $a\in X$, the set $\sigma_w(\mathrsfs{F},a)=\{\eta\in\varphi_w(\mathrsfs{F}):a\in \eta\setminus\varphi_w^{-1}(\eta)\}$ is called the set of spurious elements of $\mathcal{F}$ with respect to $a$. The set $\pi_w(\mathrsfs{F},a)=\{\eta\in\varphi_w(\mathrsfs{F}_a):\eta\men\{a\}\notin\varphi_w(\mathrsfs{F})\}$ is called the set of pure elements of $\mathcal{F}$ with respect to $a$.
  \\
  Let $\eta\in\mathcal{F}$, the set of \emph{pure elements of $\eta$}, denoted by $\pi_w(\mathrsfs{F}, \eta)$, is the set $\{a\in X: \eta\in\pi_w(a)\}$ and analogously the set of \emph{spurious elements of $\eta$} is the set $\sigma_w(\mathrsfs{F}, \eta)=\{a\in X: \eta\in\sigma_w(a)\}$.
\end{definition}
  When the $\cup$-closed set $\mathrsfs{F}$ is clear from the context, we drop $\mathrsfs{F}$ from $\sigma_w(\mathrsfs{F},a), \sigma_w(\mathrsfs{F},\eta),\pi_w(\mathrsfs{F},a), \pi_w(\mathrsfs{F},\eta)$ and we use instead $\sigma_w(a)$, $\sigma_w(\eta)$, $\pi_w(a)$,  $\pi_w(\eta)$.
\noindent We have the following lemma.
\begin{lemma}\label{lem:partition of mathcal F_a}
 The two sets $\varphi_w(\mathrsfs{F}_a)$, $\sigma_w(a)$ form a partition of $\mathcal{F}_a$. In turn $\varphi_w(\mathrsfs{F}_a)$ is partitioned by $\pi_w(a)$, $\psi(\mathcal{F}_{\overline{a}})$ where $\psi(z)=z\cup\{a\}$. Moreover $\sigma_w(a)\cup\pi_w(a)=\{z\in\mathcal{F}:z\men\{a\}\notin\mathcal{F}\}$ and $$
  |\mathcal{F}_a|=|\mathcal{F}_{\overline{a}}|+|\pi_w(a)|+|\sigma_w(a)|.
 $$
\end{lemma}
\begin{proof}
  Since $\sigma_w(a)\sub\mathcal{F}_a$ and $\varphi_w(\mathrsfs{F}_a)\sub\mathcal{F}_a$, then $\mathcal{F}_a\men\varphi_w(\mathrsfs{F}_a)$ is formed by elements $z\in\mathcal{F}_a$ for which $a$ is a spurious element of $z$, i.e. $\mathcal{F}_a\men\varphi_w(\mathrsfs{F}_a)=\sigma_w(a)$.
  By Lemma \ref{lem:embedding of the elements no having a} $\psi(\mathcal{F}_{\overline{a}})\sub\varphi_w(\mathrsfs{F}_a)$ and if $z\in\varphi_w(\mathrsfs{F}_a)\men\psi(\mathcal{F}_{\overline{a}})$ then $z\men\{a\}\notin\mathcal{F}$, otherwise $z=\psi(z\men\{a\})$. Therefore $\pi_w(a)=\varphi_w(\mathrsfs{F}_a)\men\psi(\mathcal{F}_{\overline{a}})$.
  By the previous statements it is also evident that:
  $$
  \sigma_w(a)\cup\pi_w(a)=\mathcal{F}_a\men\psi(\mathcal{F}_{\overline{a}})=\{z\in\mathcal{F}:z\men\{a\}\notin\mathcal{F}\}
  $$
  Since $\mathcal{F}_a$ is partitioned into the two sets  and $\sigma_w(a)$ $\varphi_w(\mathrsfs{F}_a)$ which in turn is partition by the two sets $\pi_w(a)$, $\psi(\mathcal{F}_{\overline{a}})$, and $\psi$ is an injective map we have:
  $$
  |\mathcal{F}_a|=|\psi(\mathcal{F}_{\overline{a}})|+|\sigma_w(a)|+|\pi_w(a)|=
  |\mathcal{F}_{\overline{a}}|+|\sigma_w(a)|+|\pi_w(a)|
  $$
  and this completes the proof of the lemma.
\end{proof}
The following proposition gives an alternative formulation of Frankl's conjecture which is the analogous of Proposition \ref{prop:another caracherization without filter}.
\begin{proposition}\label{prop:alternative formulation}
 For any $a\in X$
  $$
  |\pi_w(a)|-|\sigma_w(a)|=|P(\mathrsfs{F},a)|-|S(\mathrsfs{F},a)|
  $$
  and so Frankl's conjecture holds for $\mathrsfs{F}$ if and only if $
  |\pi_w(a)|\ge |\sigma_w(a)|$ for some $a\in X$. Moreover $|\pi_w(a)|\le |P(\mathrsfs{F},a)|$, $|\sigma_w(a)|\le |S(\mathrsfs{F},a)|$.
\end{proposition}
\begin{proof}
  It is not difficult to check that $|\mathrsfs{F}_{\overline{a}}|-| \sigma_w(a)|=|\mathcal{F}_{\overline{a}}|$ and by Lemma \ref{lem:partition of mathcal F_a} we have $|\mathcal{F}_{a}|=| \mathrsfs{F}_a|+|\sigma_w(a)|$.
  Thus by the same Lemma \ref{lem:partition of mathcal F_a} we get
  $$
  |\mathrsfs{F}_a|=|\mathrsfs{F}_{\overline{a}}|+|\pi_w(a)|-|\sigma_w(a)|
  $$
  and so, by Proposition \ref{prop:another caracherization without filter} we get the statement $|P(\mathrsfs{F},a)|-|S(\mathrsfs{F},a)|=|\mathrsfs{F}_a|-|\mathrsfs{F}_{\overline{a}}|=
  |\pi_w(a)|-|\sigma_w(a)|$.
  \\
  Let us prove the last statement showing that $\sigma_w(a)\sub\varphi_w(S(\mathrsfs{F},a))$. Let $\eta\in\sigma_w(a)$. Reasoning by contradiction, suppose that $z=\varphi_w^{-1}(\eta)\notin S(\mathrsfs{F},a)$ and so $z\cup\{a\}\in\mathrsfs{F}$. Since $z\cup\{a\}\sub \eta$, by Corollary \ref{cor:ideal correspondence} we get $\eta\in\mathcal{F}[z\cup\{a\}]\backsimeq\mathrsfs{F}[z\cup\{a\}]$ and so $a\in z\cup\{a\}\subseteq\varphi_w^{-1}(\eta)=z$ which contradicts $\eta\in\sigma_w(a)$. The statement $|\pi_w(a)|\le |P(\mathrsfs{F},a)|$ is a consequence of $|\pi_w(a)|-|\sigma_w(a)|=|P(\mathrsfs{F},a)|-|S(\mathrsfs{F},a)|$ and $|\sigma_w(a)|\le |S(\mathrsfs{F},a)|$.
\end{proof}
In view of Proposition \ref{prop:alternative formulation} it is interesting to give a lower bound to the set $|\pi_w(a)|$. The following proposition gives a partial answer, we recall that $\circ^*$ is the operator introduced in Section \ref{sec:permutation dependancy}.
\begin{proposition}\label{prop:lower bound pure}
For any $a\in X$ we have:
  $$
  \{\varphi_w(g): g\in\mathrsfs{F}_a, (g\setminus\{a\})^* = \emptyset \}\sub\pi_w(a)
  $$
\end{proposition}
\begin{proof}
Let $g\in\mathrsfs{F}_a, (g\setminus\{a\})^* = \emptyset$ and suppose, contrary to the statement, that $\varphi_w(g)\setminus\{a\}\in \mathcal{F}$. By Corollary \ref{cor:ideal correspondence} we have
$$
\varphi_w^{-1}(\varphi_w(g)\setminus\{a\}) = \bigcup_{f\subseteq \varphi_w(g)\setminus\{a\}} f \subseteq \bigcup_{f\subseteq g\setminus\{a\}} f = (g\setminus\{a\})^*
$$
whence $(g\setminus\{a\})^* \neq\emptyset$, a contradiction.
\end{proof}
We remark that the set $\{g\in \mathrsfs{F}_a: (g\setminus\{a\})^* = \emptyset\}$ is non-empty since it contains $\min(\mathrsfs{F})_a$.
\\
The subsets $\pi_w(\eta),\sigma_{w}(\eta)$ introduced in Definition \ref{def:spurious and pure elements} are the ``local" version of $\pi_w(a),\sigma_w(a)$ in the following sense:
$$
\sum_{a\in X}|\pi_w(a)|=\sum_{\eta\in\mathcal{F}}|\pi_w(\eta)|, \; \sum_{a\in X}|\sigma_w(a)|=\sum_{\eta\in\mathcal{F}}|\sigma_{w}(\eta)|
$$
We also note that by Lemma \ref{lem:partition of mathcal F_a} $\pi_w(\eta),\sigma_{w}(\eta)$ are two disjoint subsets of $\eta$ and in particular by the definition we get $\sigma_{w}(\eta)=\eta\men\varphi_w^{-1}(\eta)$. The interest in introducing such subsets is given by the following characterization:
\begin{proposition}\label{prop:characterization of pi sigma}
  For any $\eta\in\mathcal{F}$ we have:
  $$
  \sigma_{w}(\eta)=\bigcap_{\xi\sub\eta}\sigma_{w}(\xi),\;\pi_w(\eta)=
  \bigcap_{\xi\sub\eta}\xi\cap\varphi_w^{-1}(\eta)
  $$
\end{proposition}
\begin{proof}
  By Lemma \ref{lem:partition of mathcal F_a}, $\sigma_w(a)\cup\pi_w(a)=\{z\in\mathcal{F}:z\men\{a\}\notin\mathcal{F}\}$, thus it is straightforward to check
  $$
  \pi_w(\eta)\cup\sigma_{w}(\eta)=\{a\in X:\eta\men\{a\}\notin\mathcal{F}\}=\bigcap_{\xi\sub\eta}\xi
  $$
  Since $\pi_w(\eta)\sub\varphi_w^{-1}(\eta)$ and $\sigma_{w}(\eta)=\eta\men\varphi_w^{-1}(\eta)$ then
  $$
  \pi_w(\eta)=\bigcap_{\xi\sub\eta}\xi\cap\varphi_w^{-1}(\eta),\;\;
  \sigma_{w}(\eta)=\bigcap_{\xi\sub\eta}\xi\cap\sigma_{w}(\eta)
  $$
  We claim that if $\xi\sub\eta$ then $\sigma_{w}(\eta)\sub\sigma_{w}(\xi)$ from which it follows $\sigma_{w}(\eta)=\bigcap_{\xi\sub\eta}\sigma_{w}(\xi)$. Indeed by Lemma \ref{lem:embedding of the elements no having a} for all $b\in\eta\men\xi$, $b\in\varphi_w^{-1}(\xi\cup\{b\})$. Thus, since $\xi\cup\{b\}\sub\eta$, by Theorem \ref{theo:ideal correspondence}, $b\in\varphi_w^{-1}(\eta)$, whence $\eta\men\xi\sub\varphi_w^{-1}(\eta)$. By Theorem \ref{theo:ideal correspondence} we also get $\varphi_w^{-1}(\xi)\sub\varphi_w^{-1}(\eta)$, thus $\eta\men\xi\cup\varphi_w^{-1}(\xi)\sub\varphi_w^{-1}(\eta)$ from which we obtain $\sigma_{w}(\eta)\sub\sigma_{w}(\xi)$.
\end{proof}
If $f\subsetneq g$ for some $f,g \in\mathrsfs{F}$, then in general $\sigma_w(\varphi_w(f))\subsetneq \sigma_w(\varphi_w(g))$ do not hold. However if we keep the freedom to choose the order of the rising we can have this property. With the notation of Section \ref{sec:permutation dependancy} we have the following:
\begin{lemma}\label{lem:containing spurious}
  Let $f,g\in\mathrsfs{F}$ with $f\subseteq g$ and let $\eta\in\max(Fib(g))$, then there is a word $w'=a_{i_1}\ldots a_{i_n}$ such that $\eta = \varphi_{w'}(g)$ and
  $$
  \sigma_{w'}(\varphi_{w'}(g))\subseteq \sigma_{w'}(\varphi_{w'}(f))
  $$
\end{lemma}
\begin{proof}
Let us prove that $\eta' =(\eta\setminus g)\cup f\in Fib(f)$. It is obvious that $f\subseteq\eta'$, a let us assume, contrary to our claim, that there is $h\in\mathrsfs{F}$ such that $h\subseteq\eta'$ with $f\subsetneq h$. Thus $(h\setminus f)\cap (\eta\setminus g) \neq \emptyset$. Since $\eta'\sub\eta$, then $h\subseteq\eta$ and so $h\subseteq \eta^*=g$. In particular we have $(h\setminus f)\subseteq g$ which contradicts $(h\setminus f)\cap (\eta\setminus g) \neq \emptyset$. Therefore $\eta'\in Fib(f)$, and let $\nu\in\max(Fib(f))$ such that $\eta'\subseteq\nu$. Then we have
\begin{equation}\label{eq: inclusion elements inclusion sigma's}
  \eta\men\eta^* = \eta\men g = \eta'\men f \sub \nu\men f = \nu\men \nu^*
\end{equation}
If we prove that there is a word $w'$ such that $\eta=\varphi_{w'}(g), \nu= \varphi_{w'}(f)$ then we have proved the statement of the lemma since (\ref{eq: inclusion elements inclusion sigma's}) holds and $\sigma_{w'}(\eta) = \eta\men\eta^*, \sigma_{w'}(\nu) = \nu\men\nu^*$. Since $\eta\men\eta^*\sub\nu\men\nu^*$, then we can find a word $w'$ such that both $\eta\men\eta^*$ and $\nu\men\nu^*$ are contained in a prefix of $w'$, hence by Lemma \ref{lem:word generatin a maximal fiber}, we have $\eta=\varphi_{w'}(g), \nu= \varphi_{w'}(f)$.
\end{proof}

\subsection{The average length}
The \emph{average of the length} of the elements of $\mathrsfs{F}$, simply the \emph{average} of the family $\mathrsfs{F}$, is the integer $\frac{1}{|\mathrsfs{F}|}\sum_{f\in\mathrsfs{F}}|f|$, this number is important because the following well known equality holds
$$
\sum_{a\in X}\frac{|\mathrsfs{F}_a|}{|\mathrsfs{F}|}=\frac{1}{|\mathrsfs{F}|}\sum_{f\in\mathrsfs{F}}|f|
$$
For instance the averaged Frankl's property $\frac{1}{|\mathrsfs{F}|}\sum_{f\in\mathrsfs{F}}|f|\ge\frac{n}{2}$ implies that Frankl's conjecture is true for $\mathrsfs{F}$. Unfortunately the converse is not true, indeed it is a well know fact that many union-closed families fail to satisfy the averaged Frankl's property (see \cite{Czedeli1,CzedeliMaroti}). However the average of $\mathrsfs{F}$ is still an interesting parameter at least because any lower bound on it gives rise to a lower bound of $
\max_{a\in X}\{|\mathrsfs{F}_a|/|\mathrsfs{F}|\}$.
In \cite{Reimer} Reimer shows that
$$
\frac{1}{|\mathrsfs{F}|}\sum_{f\in\mathrsfs{F}}|f|\ge \frac{1}{2}\log_2(|\mathrsfs{F}|)
$$
and in \cite{Falgas} the bound is improved in the case of a separating family.
What we consider here is the localized version of the average of $\mathrsfs{F}$, given a subfamily $S\sub \mathrsfs{F}$, the \emph{average of $\mathrsfs{F}$ localized on $S$} is defined by
$$
\frac{1}{|\mathrsfs{F}[S]|}\sum_{f\in\mathrsfs{F}[S]}|f|
$$
and gives the average of the length of the elements contained in the principal ideal of $\mathrsfs{F}$ generated by $S$. Our aim is to provide lower bounds to such quantity. Note that we can assume without loss of generality that $S$ is an antichain. We fix the notation and for the rest of the section $\mathrsfs{F}$ denotes a $\cup$-closed family of sets of $X=\{a_1,\ldots,a_n\}$, $S\sub\mathrsfs{F}$ is an antichain, and $\mathcal{F}=\varphi_w(\mathrsfs{F})$ is the upward-closed family associate to $\mathrsfs{F}$ with respect to the word $w=a_1a_2\ldots a_n$.
\begin{proposition}\label{prop:local averaging general fact}
The following bound holds:
  $$
  \sum_{f\in\mathrsfs{F}[S]}|f|\ge\frac{n}{2}|\mathrsfs{F}[S]|+\frac{1}{2}\sum_{a\in X}|\pi_w(a)\cap S^{\uparrow}|-|\sigma_w(a)\cap S^{\uparrow}|.
  $$
  with equality if $S=\min(\mathrsfs{F})$.
\end{proposition}
\begin{proof}
  By Lemma \ref{lem:partition of mathcal F_a} there is a partition $\mathcal{F}_a=\psi(\mathcal{F}_{\overline{a}})\cup\pi_w(a)\cup\sigma_w(a)$, hence:
  \begin{equation}\label{eq: balance F_a[S]}
  \mathcal{F}_a[S]=(\psi(\mathcal{F}_{\overline{a}})\cap S^{\uparrow})\cup(\pi_w(a)\cap S^{\uparrow})\cup(\sigma_w(a)\cap S^{\uparrow})
  \end{equation}
  We have $\psi(\mathcal{F}_{\overline{a}}\cap S^{\uparrow})\subseteq \psi(\mathcal{F}_{\overline{a}})\cap S^{\uparrow}$ with equality if $S=\min(\mathrsfs{F})$, whence $\sum_a|\psi(\mathcal{F}_{\overline{a}})\cap S^{\uparrow}|\ge\sum_a|\psi(\mathcal{F}_{\overline{a}}\cap S^{\uparrow})|=\sum_{\eta\in\mathcal{F}[S]}(n-|\eta|)$. Thus summing all the equalities (\ref{eq: balance F_a[S]}) on the index $a\in X$, we get
  \begin{equation}\label{eq: 1 evaraging theorema }
    2\sum_{\eta\in\mathcal{F}[S]}|\eta|\ge n|\mathcal{F}[S]|+\sum_{a\in X}|\pi_w(a)\cap S^{\uparrow}|+|\sigma_w(a)\cap S^{\uparrow}|.
  \end{equation}
  By Theorem \ref{theo:ideal correspondence}, $\sigma_w(a)\cap S^{\uparrow}=\{\varphi_w(f), f\in\mathrsfs{F}[S], a\in\varphi_w(f)\men f\}$, and so:
  $$
  \sum_{a\in X} |\sigma_w(a)\cap S^{\uparrow}|=\sum_{f\in\mathrsfs{F}[S]}(|\varphi_w(f)|-|f|)
  $$
  Moreover by Theorem \ref{theo:ideal correspondence} we also get
  \begin{align}
  \nonumber &\sum_{f\in\mathrsfs{F}[S]}|f|=\sum_{\eta\in\mathcal{F}[S]}|\varphi_w^{-1}(\eta)|=
  \sum_{\eta\in\mathcal{F}[S]}|\eta|-|\eta\men\varphi_w^{-1}(\eta)|=\\
  \nonumber &\sum_{\eta\in\mathcal{F}[S]}|\eta|-\sum_{f\in\mathrsfs{F}[S]}(|\varphi_w(f)|-|f|)= \sum_{\eta\in\mathcal{F}[S]}|\eta|-\sum_{a\in X}|\sigma_w(a)\cap S^{\uparrow}|
  \end{align}
  Therefore using (\ref{eq: 1 evaraging theorema }) and $\mathcal{F}[S]\simeq \mathrsfs{F}[S]$ (Corollary \ref{cor:ideal correspondence}) we get
  $$
  \sum_{f\in\mathrsfs{F}[S]}|f|\ge\frac{n}{2}|\mathrsfs{F}[S]|+\frac{1}{2}\sum_{a\in X}|\pi_w(a)\cap S^{\uparrow}|-|\sigma_w(a)\cap S^{\uparrow}|
  $$
  with equality if $S=\min(\mathrsfs{F})$.
\end{proof}
We have the following corollary on the local average in the case $\min(\mathrsfs{F})$ is a maximal antichain and the elements are uniformly bounded by some integer.
\begin{corollary}\label{cor: bound average maximal}
  Let $\mathrsfs{F}$ be a $\cup$-closed family of sets such that $\mathcal{G}=\min(\mathrsfs{F})$ is a maximal antichain of $2^X$ and there is a positive integer $k$ such that for all $g\in \mathcal{G}$, $|g|\le k$, then:
  $$
  \frac{1}{|\mathrsfs{F}[S]|}\sum_{f\in\mathrsfs{F}[S]}|f|\ge\frac{n-k}{2}+\frac{1}{2|\mathrsfs{F}[S]|}\sum_{a\in X}|\pi_w(a)\cap S^{\uparrow}|.
  $$
\end{corollary}
\begin{proof}
  We have already noted in the proof of Proposition \ref{prop:local averaging general fact} that:
  $$
  \sum_{a\in X} |\sigma_w(a)\cap S^{\uparrow}|=\sum_{f\in\mathrsfs{F}[S]}(|\varphi_w(f)|-|f|)=\sum_{f\in\mathrsfs{F}[S]}|\varphi_w(f)\men f|
  $$
  by the same proposition it is sufficient to prove that $|\varphi_w(f)\men f|\le k$. Since $\mathcal{G}$ is a maximal antichain, then for any $f\in\mathrsfs{F}[S]$ there is a $g\in \mathcal{G}$ such that either $g\sub \varphi_w(f)\men f$ or $\varphi_w(f)\men f\sub g$. We prove that only $\varphi_w(f)\men f\sub g$ can occur, and so $|\varphi_w(f)\men f|\le k$. Indeed, if $g\sub \varphi_w(f)\men f$, then $g\sub \varphi_w(f)$ and so, by Theorem \ref{theo:ideal correspondence}, $g\sub f$, a contradiction.
\end{proof}
\noindent Observe that Corollary \ref{cor: bound average maximal} also holds if we assume the existence of a maximal antichain $\mathcal{A}\subseteq \mathrsfs{F}$ such that $|g|\le k$ for all $g\in \mathcal{A}$.
\\
The following corollary is the analogous of Corollary \ref{cor: bound average maximal} in the case we drop the maximality condition. Let $S\sub\mathrsfs{F}$ we define $\sigma(S)=\max\{|\sigma_{w\theta}(f)|:f\in S, \theta\in\mathfrak{G}_X\}$.
\begin{corollary}\label{cor: bound average using invariant}
    $$
    \frac{1}{|\mathrsfs{F}[S]|}\sum_{f\in\mathrsfs{F}[S]}|f|\ge\frac{n-\sigma(S)}{2}+\frac{1}{2|\mathrsfs{F}[S]|}\sum_{a\in X}|\pi_w(a)\cap S^{\uparrow}|.
    $$
\end{corollary}
\begin{proof}
  Like in the proof of Corollary \ref{cor: bound average maximal} and
  by Proposition \ref{prop:local averaging general fact} it is sufficient to show $|\varphi_w(f)\men f|=|\sigma_w(\varphi_w(g))|\le \sigma(S)$ for all $g\in \mathrsfs{F}[S]$. Consider any $g\in \mathrsfs{F}[S]$, and let $f\in S$ such that $f\subseteq g$. By Lemma \ref{lem:containing spurious} there is a word $w'$ such that
  $$
  \sigma_w(\varphi_w(g)) = \varphi_w(g)\men g = \varphi_{w'}(g)\men g \subseteq \varphi_{w'}(f)\men f = \sigma_{w'}(\varphi_{w'}(f))
  $$
  and so the claim $|\sigma_w(\varphi_w(g))|\le \sigma(S)$.
\end{proof}
The following theorem gives a lower bound of the average localized on $S$ depending on the parameter $|S^{\uparrow}|$.
\begin{theorem}\label{theo: general lower bound}
$$
  \frac{1}{|\mathrsfs{F}[S]|}\sum_{f\in\mathrsfs{F}[S]}|f|\ge\frac{n}{2}+
  \frac{1}{2|\mathrsfs{F}[S]|}\sum_{a\in X}|\pi_w(a)\cap S^{\uparrow}|-\frac{1}{2}\log_2\biggl\{\frac{|S^{\uparrow}|}{|\mathrsfs{F}[S]|}\biggl\}
$$
and the bound is attained when $S=\min(\mathrsfs{F})$ and when $\mathrsfs{F}$ is upward-closed.
\end{theorem}
\begin{proof}
  By Proposition \ref{prop:local averaging general fact} it is enough to give an upper bound to the quantity $\sum_{a\in X}|\sigma_w(a)\cap S^{\uparrow}|$. Following a similar argument in \cite{Reimer}, we use Jensen's inequality to upper bound $\sum_{a\in X}|\sigma_w(a)\cap S^{\uparrow}|=\sum_{f\in\mathrsfs{F}[S]}(|\varphi_w(f)|-|f|)$. Indeed, we have
  $$
  \exp_2\biggl\{\frac{1}{|\mathrsfs{F}[S]|}\sum_{f\in\mathrsfs{F}[S]}(|\varphi_w(f)|-|f|)\biggl\}\le \frac{1}{|\mathrsfs{F}[S]|}\sum_{f\in\mathrsfs{F}[S]}2^{|\varphi_w(f)|-|f|}.
  $$
  By Proposition \ref{prop:interval property}, $f\neq g$ implies $[f,\varphi_w(f)]\cap [g,\varphi_w(g)]=\emptyset$, hence since $|[f,\varphi_w(f)]|=2^{|\varphi_w(f)|-|f|}$ we get
  $$
  \frac{1}{|\mathrsfs{F}[S]|}\sum_{f\in\mathrsfs{F}[S]}(|\varphi_w(f)|-|f|)\le \log_2\biggl\{\frac{|C(\mathcal{F},S)|}{|\mathrsfs{F}[S]|}\biggl\}.
  $$
  where $C(\mathcal{F},S)=\bigcup_{f\in\mathrsfs{F}[S]}[f,\varphi_w(f)]$. The statement of the theorem thus follows from $C(\mathcal{F},S)\sub S^{\uparrow}$.
  \\
  If $S=\min(\mathrsfs{F})$ we have the equality in the bound of Proposition \ref{prop:local averaging general fact}, moreover if $\mathrsfs{F}$ is upward-closed, then $\mathcal{F}=\mathrsfs{F}$. Thus $\sum_{a\in X}|\sigma_w(a)\cap S^{\uparrow}|=0$ which is equal to $\frac{1}{2}\log_2\{|S^{\uparrow}|/|\mathrsfs{F}[S]|\}$ since $\mathrsfs{F}$ is upward-closed and so $S^{\uparrow}=\mathrsfs{F}[S]$. Therefore the lower bound in the statement is reached for $S=\min(\mathrsfs{F})$ and the class of upward-closed families.
\end{proof}
\begin{rem}\label{rem: remark tightness pure}
  In Theorem \ref{theo: general lower bound}, Corollaries \ref{cor: bound average using invariant},\ref{cor: bound average maximal}, we can give a lower bound to the quantity $\frac{1}{2|\mathrsfs{F}[S]|}\sum_{a\in X}|\pi_w(a)\cap S^{\uparrow}|$. Indeed, by Proposition \ref{prop:lower bound pure} and Theorem \ref{theo:ideal correspondence} it is not difficult to see that
  $$
  \sum_{a\in X}|\pi_w(a)\cap S^{\uparrow}|\ge \sum_{g\in \mathrsfs{F}[S] }|\{a\in g: (g\setminus\{a\})^* = \emptyset\}|
  $$
  and the equality is attained if $S=\min(\mathrsfs{F})$ and when $\mathrsfs{F}$ is upward-closed.
\end{rem}

\subsection{The invariant case}
In this section we obtain some lower bounds on the average of $\mathrsfs{F}$ localized on $S$ using the invariant upward-closed set associated to $\mathrsfs{F}$. The following definition can be considered as the analogous of the spurious and pure elements of Definition \ref{def:spurious and pure elements} in the invariant case.
\begin{definition}
  Let $\textbf{U}(\mathrsfs{F})$ be the invariant upward-closed set associated to $\mathrsfs{F}$ and let $a\in X$ the set
  $$
  \Sigma(\mathrsfs{F}, a) = \{g\in \mathrsfs{F}_{\overline{a}}:\forall\eta\in \max(Fib(g)), a\in\eta\}
  $$
  is called the set of hyper-spurious elements. The local version of this set is $\Sigma(\mathrsfs{F}, g) = \{a\in X\setminus g: \forall\eta\in\max(Fib(g)), a\in\eta \}$. The elements of the set
  $$
  \Pi(\mathrsfs{F}, a) = \{g\in\mathrsfs{F}_a: \forall \eta\in \max(Fib(g)), \eta\setminus\{a\}\notin \textbf{U}(\mathrsfs{F})\}
  $$
  are called hyper-pure. The local version of this set is $\Pi(\mathrsfs{F}, g) = \{a\in g: \forall \eta\in \max(Fib(g)), \eta\setminus\{a\}\notin \textbf{U}(\mathrsfs{F})\}$.
\end{definition}
The connection between these sets and the spurious, pure sets introduced in Definition \ref{def:spurious and pure elements} is given by the following proposition.
\begin{proposition}\label{prop: characte hyper invariant case}
The following equalities hold:
\begin{equation}\label{eq:hyper-spurious}
  \Sigma(\mathrsfs{F}, g) = \bigcap_{\eta\in\max(Fib(g))}\eta\setminus g = \bigcap_{\theta\in \mathfrak{S}_X}\sigma_{w\theta}(\mathrsfs{F}, \varphi_{w\theta}(g))
\end{equation}
\begin{equation}\label{eq:hyper-pure}
  \Pi(\mathrsfs{F}, g)  = \bigcap_{\theta\in \mathfrak{S}_X}\pi_{w\theta}(\mathrsfs{F}, \varphi_{w\theta}(g))
\end{equation}
\end{proposition}
\begin{proof}
  The first equality in (\ref{eq:hyper-spurious}) is a consequence of the definition, the second one of Theorem \ref{theo: characterization orbits}. Let us prove (\ref{eq:hyper-pure}). Let $b\in \Pi(\mathrsfs{F}, g) $, then for any $\eta\in \max(Fib(g))=\mathfrak{S}_X\cdot \varphi_{w}(g)$ (by Theorem \ref{theo: characterization orbits}) we have $\eta\setminus\{b\}\notin \textbf{U}(\mathrsfs{F})$, hence for any $\theta\in \mathfrak{S}_X$, $\varphi_{w\theta}(g)\setminus\{b\}\notin\varphi_{w\theta}(\mathrsfs{F})$, i.e. $b\in \bigcap_{\theta\in \mathfrak{S}_X}\pi_{w\theta}(\mathrsfs{F}, \varphi_{w\theta}(g))$. On the other side, let $b\in \bigcap_{\theta\in \mathfrak{S}_X}$ $\pi_{w\theta}(\mathrsfs{F}, \varphi_{w\theta}(g))$.
  To obtain a contradiction suppose that there is $\eta\in\max(Fib(g))$, for some $g\in \mathrsfs{F}$, such that $\eta\setminus\{b\}\in \textbf{U}(\mathrsfs{F})$, say $\eta\setminus\{b\}\in \max(Fib(h))$ for some $h\in \mathrsfs{F}$. Since by Theorem \ref{theo: characterization orbits} $\max(Fib(h)) = \mathfrak{S}_X\cdot \varphi_{w}(h)$, there is a $\vartheta\in\mathfrak{S}_X $ such that $\eta\setminus\{b\}=\varphi_{w\vartheta}(h)$. Since $\eta\setminus\{b\}\subseteq \eta$ we have $\eta\in\varphi_{w\vartheta}(\mathrsfs{F})$, in particular, since $\varphi_{w\vartheta}^{-1}(\eta) = \eta^* = g$, we have $\eta = \varphi_{w\vartheta}(g)$. However $b\in \bigcap_{\theta\in \mathfrak{S}_X}\pi_{w\theta}(\mathrsfs{F}, \varphi_{w\theta}(g))$ implies $b\in \pi_{w\vartheta}(\mathrsfs{F}, \eta)$ which contradicts $\eta\setminus\{b\} \in \varphi_{w\vartheta}(\mathrsfs{F})$.
\end{proof}
We recall that at the end of Section \ref{sec:permutation dependancy} we have introduced the set
$$
    Cov_a(g) = \{h\in \mathrsfs{F}_a: (h\setminus\{a\})^* = g\}
$$
we have the following proposition:
\begin{proposition}\label{prop: invariant hyper-pure, hyper-spu}
  $$
  \Sigma(\mathrsfs{F}, g) = \{b\in X\setminus g: Cov_b(g) = \emptyset\} = X\setminus \bigcup_{\{h:g\lessdot h\}} h
  $$
    $$
  |\mathrsfs{F}_{\overline{a}}| - |\Sigma(\mathrsfs{F}, a)|\le \sum_{g\in \mathrsfs{F}_{\overline{a}}\setminus\Sigma(\mathrsfs{F}, a)} |\max(Cov_a(g))|\le |\mathrsfs{F}_{a}| - |\Pi(\mathrsfs{F}, a)|, \forall a\in X
    $$
\end{proposition}
\begin{proof}
  By Proposition \ref{prop: covered in a} and Lemma \ref{lem: characterization covering} we have that $Cov_a(g) =\emptyset$ if and only if $a\in \Sigma(\mathrsfs{F}, g)$. The second equality is also a consequence of Proposition \ref{prop: covered in a} and the definitions. Let us prove the inequalities.
  We first claim that for any $h\in \max(Cov_a(g))$ with $g\in\mathrsfs{F}_{\overline{a}}\setminus\Sigma(\mathrsfs{F}, a)$ and for any $\eta\in\max(Fib(h))$, $\eta\setminus\{a\}\in \textbf{U}(\mathrsfs{F})$. Since $(h\setminus\{a\})^* = g$ and $(h\setminus\{a\})\subseteq \eta\setminus\{a\}$ we have $g\subseteq (\eta\setminus\{a\})^*$. On the other hand, let $(\eta\setminus\{a\})^* = g'$. Since $g' \subseteq \eta\setminus\{a\}$ and $g'\subseteq \eta^* = h$, then $g'\subseteq (h\setminus\{a\})^* = g$ from which we have the equality $(\eta\setminus\{a\})^* = g$.
  Therefore $(\eta\setminus\{a\})\in Fib(g)$, and so there is a $\nu\in\max(Fib(g))$ with $(\eta\setminus\{a\})\subseteq\nu$. Consider $\nu\cup\{a\}$ and let us prove that $\nu\cup\{a\}\in\max(Fib(h))$. Clearly $\nu\cup\{a\}\in\max(Fib(h'))$ for some $h'$, we observe that since $h\subseteq \eta\subseteq\nu\cup\{a\}$ we have $h\subseteq (\nu\cup\{a\})^* = h'$. If we prove that $h'\in Cov_a(g)$, then by the maximality of $h$, we get $h=h'$. Suppose, contrary to our claim, that $h'\notin Cov_a(g)$. Thus, if we put $g'=(h'\setminus\{a\})^*$, we have $g = (h\setminus\{a\})^*\subsetneq (h'\setminus\{a\})^* = g'$. However, we also have $g'= (h'\setminus\{a\})^*\subseteq \nu^* = g$, a contradiction.
  Therefore, the claim is true and so we can deduce the following inclusion:
  $$
    \bigcup_{g\in \mathrsfs{F}_{\overline{a}}\setminus\Sigma(\mathrsfs{F}, a)} \max(Cov_a(g)) \subseteq \mathrsfs{F}_{a}\setminus\Pi(\mathrsfs{F}, a)
  $$
  Thus, the inequality easily follows from this inclusion and the following facts: $Cov_a(g)\neq\emptyset$ iff $g\in \mathrsfs{F}_{\overline{a}}\setminus\Sigma(\mathrsfs{F}, a)$ and $Cov_a(g)\cap Cov_a(g') = \emptyset$ for $g\neq g'$.
\end{proof}
The following theorem is the analogous of Theorem \ref{theo: general lower bound} in the invariant case.
\begin{proposition}\label{prop: lower bound average invariant case}
  $$
  \frac{1}{|\mathrsfs{F}[S]|}\sum_{f\in\mathrsfs{F}[S]}|f|\ge \frac{n}{2} - \frac{1}{2|\mathrsfs{F}[S]|}\sum_{f\in\mathrsfs{F}[S]} |\Sigma(\mathrsfs{F},g)| + \frac{1}{2|\mathrsfs{F}[S]|}\sum_{a\in X}|\Pi(\mathrsfs{F},a)\cap S^{\uparrow}|
  $$
  this bound is attained for $S=\min(\mathrsfs{F})$ and when $\mathrsfs{F}$ is upward-closed.
\end{proposition}
\begin{proof}
Proposition \ref{prop: invariant hyper-pure, hyper-spu} can be easily adapt to prove that for all $a\in X$:
  $$
  |\mathrsfs{F}_{\overline{a}}[S]|-|\Sigma(\mathrsfs{F}, a)\cap S^{\uparrow}| \le |\mathrsfs{F}_{a}[S]|- |\Pi(\mathrsfs{F}, a)\cap S^{\uparrow}|
  $$
  The statement can be thus proved summing all these inequalities with $a$ running on $X$ and using the equalities $\sum_{a\in X}|\mathrsfs{F}_{a}[S]| = \sum_{f\in\mathrsfs{F}_{a}[S]}|f|$, $\sum_{a\in X}|\mathrsfs{F}_{\overline{a}}[S]| = \sum_{f\in\mathrsfs{F}_{a}[S]}(n - |f|)$, $\sum_{a\in X}|\Sigma(\mathrsfs{F}, a)\cap S^{\uparrow}|$
  $= \sum_{f\in\mathrsfs{F}_{a}[S]} |\Sigma(\mathrsfs{F},f)|$.
  \\
  By Theorem \ref{theo: general lower bound} and Remark \ref{rem: remark tightness pure} in the case of an upward-closed family $\mathrsfs{F}$ and $S=\min(\mathrsfs{F})$, we have
  $$
  \frac{1}{|\mathrsfs{F}[S]|}\sum_{f\in\mathrsfs{F}[S]}|f| = \frac{n}{2} + \frac{1}{2|\mathrsfs{F}[S]|}\sum_{g\in \mathrsfs{F}[S] }|\{a\in g: (g\setminus\{a\})^* = \emptyset\}|
  $$
  On the other hand by Proposition \ref{prop: invariant hyper-pure, hyper-spu} it is not difficult to check that in the case $\mathrsfs{F}$ is an upward-closed family $\Sigma(\mathrsfs{F},g) = \emptyset$ and $\Pi(\mathrsfs{F},g) = \{a\in g: (g\setminus\{a\})^* = \emptyset\}$ and so the bound is attained in this case.
\end{proof}
Using the first equality of Proposition \ref{prop: invariant hyper-pure, hyper-spu} and Proposition \ref{prop: lower bound average invariant case} we can rewrite the bound of Proposition \ref{prop: lower bound average invariant case} as
$$
  \frac{1}{|\mathrsfs{F}[S]|}\sum_{f\in\mathrsfs{F}[S]}|f|\ge \frac{1}{2|\mathrsfs{F}[S]|}\sum_{f\in\mathrsfs{F}[S]} |\bigcup_{\{h:f\lessdot h\}} h| + \frac{1}{2|\mathrsfs{F}[S]|}\sum_{a\in X}|\Pi(\mathrsfs{F},a)\cap S^{\uparrow}|
$$
We observe that by Propositions \ref{prop: characte hyper invariant case}, \ref{prop:lower bound pure}, similarly to Remark \ref{rem: remark tightness pure}, we also have the following lower bound
    $$
  \sum_{a\in X}|\Pi(\mathrsfs{F},a)\cap S^{\uparrow}|\ge \sum_{g\in \mathrsfs{F}[S] }|\{a\in g: (g\setminus\{a\})^* = \emptyset\}|
    $$

\section{Upper bounds for the join-irreducible elements of a union-closed family}\label{sec: bounf join irreducible}
Let $\mathrsfs{F}$ be a $\cup$-closed family of sets of $2^X$ with $X=\{a_1,a_2,\ldots,a_n\}$, in this section we use the techniques obtained in Section \ref{sec:II} to give an upper bound to the number of join-irreducible elements of $\mathrsfs{F}$. We remark that if $m\in J(\mathrsfs{F})$ then $\mathrsfs{F}\setminus\{m\}$ is again a $\cup$-closed family of $2^X$. Therefore it is interesting and quite natural studing the effect of erasing an irreducible elements from $\mathrsfs{F}$ in the rising process. For this reasons we will denote by $\varphi_w$, $\varphi'_w$ the rising function with respect to the word $w=a_1a_2\ldots a_n$ respectively of $\mathrsfs{F}$, $\mathrsfs{F}'=\mathrsfs{F}\setminus\{m\}$. We recall that the rising function at the $i$-th step is defined by
  $$
    \begin{array}{l}
    \varphi_{\mathrsfs{F}_{i},a_{i+1}}(z)=\left\{\begin{array}{l}
    z\cup\{a_{i+1}\}\;\text{ if } z\cup\{a_{i+1}\}\notin \mathrsfs{F}_i,\\
    z \text{ otherwise};
    \end{array}\right.
    \end{array}
  $$
Here we simplify the cumbersome notation and we write $\varphi_{a_{i+1}}$, $\varphi'_{a_{i+1}}$ for $\varphi_{\mathrsfs{F}_{i},a_{i+1}}$, $\varphi_{\mathrsfs{F}'_{i},a_{i+1}}$, respectively. With this notation, the rising function with respect to $w=a_1\ldots a_n$ is the last function $\varphi_n$ of the sequence of functions defined inductively by $\varphi_i = \varphi_{a_{i}}\circ\varphi_{i-1}$ for $i = 1,\ldots n$ where $\varphi_0$ is the identity function on $2^X$. We have the following lemma.
\begin{lemma}\label{lem:erasing element}
  With the above notation, for each $i\in\{0,1,\ldots,n\}$ there is a element $\mu_i\in\mathrsfs{F}_i$ such that $\mathrsfs{F}'_i=\mathrsfs{F}_i\setminus\{\mu_i\}$. Moreover we have two possibilities
  \begin{enumerate}
    \item if there is no $z\in \mathrsfs{F}'_i$ such that $z\cup a_{i+1}=\mu_i$, then $\mathrsfs{F}'_{i+1}=\mathrsfs{F}_{i+1}\setminus\{\mu_{i+1}\}$ with $\mu_{i+1}=\varphi_{a_{i+1}}(\mu_i)$ and $\varphi_{a_{i+1}}(z)=\varphi'_{a_{i+1}}(z)$ for all $z\in\mathrsfs{F}'_{i}$.
    \item if there is $z\in \mathrsfs{F}'_i$ such that $z\cup a_{i+1}=\mu_i$, then $\mathrsfs{F}'_{i+1}=\mathrsfs{F}_{i+1}\setminus\{\mu_{i+1}\}$ with $\mu_{i+1}=z=\varphi_{a_{i+1}}(z)$, $\varphi'_{a_{i+1}}(z)=\mu_i$ and $\varphi'_{a_{i+1}}(y)=\varphi_{a_{i+1}}(y)$ for all $y\in\mathrsfs{F}'_{i}\men\{z\}$.
  \end{enumerate}
\end{lemma}
\begin{proof}
  We prove the statement by induction on the index $i$. The statement is true for $i=0$, since $\mathrsfs{F}'_0=\mathrsfs{F}'=\mathrsfs{F}\setminus\{m\}=\mathrsfs{F}_0\setminus\{m\}$. So, putting $\mu_0=m$, we can suppose that the statement is true for $i>0$ and let us prove it for $i+1$.
   By induction, there is an element $\mu_i\in \mathrsfs{F}_i$ such that $\mathrsfs{F}'_i=\mathrsfs{F}_i\setminus\{\mu_i\}$. Let $z\in \mathrsfs{F}'_{i}$, we have the following cases:
   \begin{itemize}
    \item[i)] $z\cup\{a_{i+1}\}\notin \mathrsfs{F}_i$ and so also $z\cup\{a_{i+1}\}\notin \mathrsfs{F}'_i$ which implies $\varphi_{a_{i+1}}(z)=\varphi'_{a_{i+1}}(z)=z\cup\{a_{i+1}\}$.\\
    \item[ii)] $z\cup\{a_{i+1}\}\in\mathrsfs{F}'_i\subseteq\mathrsfs{F}_i$ and so $\varphi_{a_{i+1}}(z)=\varphi'_{a_{i+1}}(z)=z$.\\
    \item[iii)] $z\cup\{a_{i+1}\}\in\mathrsfs{F}_i\setminus\mathrsfs{F}'_i=\{\mu_i\}$, and so $z\cup\{a_{i+1}\}=\mu_i$. Hence $\varphi'_{a_{i+1}}(z)=z\cup\{a_{i+1}\}=\mu_i=\varphi_{a_{i+1}}(\mu_i)$ and $\varphi_{a_{i+1}}(z)=z$.
  \end{itemize}
  Thus, if condition $z\cup\{a_{i+1}\}=\mu_i$ do not hold for any $z\in \mathrsfs{F}_i$, then i), ii) hold and so condition 1. is true. Otherwise if there is $z\in \mathrsfs{F}'_i$ such that $z\cup a_{i+1}=\mu_i$, then iii) holds and so $z=\varphi_{a_{i+1}}(z)$ is missing in $\mathrsfs{F}'_{i+1}$, whence $\mathrsfs{F}'_{i+1}=\mathrsfs{F}_{i+1}\setminus\{\mu_{i+1}\}$ with $\mu_{i+1}=z$ and $\varphi'_{a_{i+1}}(z)=\mu_i$. For any $y\in\mathrsfs{F}'_{i}\men\{z\}$ either condition i) or ii) holds and so $\varphi'_{a_{i+1}}(y)=\varphi_{a_{i+1}}(y)$, and this concludes the proof of statement 2.
\end{proof}
The previous Lemma shows that in each $i$-section $\mathrsfs{F}'_{i}$ there is exactly one missing element belonging to $\mathrsfs{F}_{i}\men\mathrsfs{F}'_{i}$, this element plays an important role in the way the raising function changes. For this reason we call $\mu_i$ of Lemma \ref{lem:erasing element}, the \emph{missing element} at the $i$-th section. The next lemma gives a more precise description of the way the rising function changes.
\begin{lemma}[swapping lemma]\label{lem:swapping lemma}
  With the notation of Lemma \ref{lem:erasing element}, there are $k+1$ different elements $m_i\in\mathrsfs{F}$, for $i=0,\ldots,k$ such that $m_0=m$ and an increasing sequence of $k$ integers $1\le i_1<\ldots <i_k< n$ such that for all $0<j\le k$
  $$
    \begin{array}{l}
    \varphi'_{t}(m_j)=\left\{\begin{array}{l}
    \varphi_{t}(m_{j})\;\text{ if } t< i_j\\
    \varphi_{t}(m_{j-1})\text{ otherwise};
    \end{array}\right.
    \end{array}
  $$
  while $\varphi'_{t}(z)=\varphi_{t}(z)$ for all $1\le t\le n$ and $z\in\mathrsfs{F}\men\{m_0,\ldots,m_k\}$. For any $0\le i\le n$ the missing element is $\mu_i=\varphi_i(m_s)$ where $0<s\le k$ satisfies $i_s\le i< i_{s+1}$ if $s < k$ or $i_k\le i<n$ if $s = k$. Moreover for all $1\le j\le k$, $\varphi_{i_j-1}'(m_j)\cup\{a_{i_j}\}=\varphi_{i_j-1}(m_{j-1})$.
\end{lemma}
\begin{proof}
  By Lemma \ref{lem:erasing element} we have $\varphi'_{a_t}(z)=\varphi_{a_t}(z)$ for all $z\in\mathrsfs{F}$ and the missing element is $\mu_t=\varphi_t(m_0)$ for all $1\le t< i_1\le n$ where $i_1$ is the first integer such that there is an element $\varphi_{i_1-1}'(m_1)\in\mathrsfs{F}'_{i_1-1}$, for some $m_1\in\mathrsfs{F}$ with $m_1\neq m_0$, satisfying $
  \varphi_{i_1-1}'(m_1)\cup\{a_{i_1}\}=\mu_{i_1-1}=\varphi_{i_1-1}(m_0)$.
  Therefore by Lemma \ref{lem:erasing element} we have that $\varphi_{i_1}'(m_1)=\varphi_{i_1-1}(m_0)=\varphi_{i_1}(m_0)$ and the missing element becomes $\mu_{i_1}=\varphi_{i_1}(m_1)$. Moreover, since $\varphi_{i_1-1}'(m_1)=\varphi_{i_1-1}(m_1)$ and $\varphi_{i_1-1}'(m_1)\cup\{a_{i_1}\}=\mu_{i_1-1}=\varphi_{i_1-1}(m_0)\in\mathrsfs{F}_{i_1-1}$, by Lemma \ref{lem:matching property} we get $a_{i_1}\notin\varphi_{i_1}(m_1)=\mu_{i_1}$. In this way we have proved the base case of the following property
  \begin{itemize}
  \item[$P_h:$]
  There is a sequence of integers $1\le i_1<\ldots i_j\le h< n$ and $j+1$ different elements $m_0,\ldots, m_j\in\mathrsfs{F}$ such that for all $0<l\le j$ and for all $t\le h$
  $$
    \begin{array}{l}
    \varphi'_{t}(m_l)=\left\{\begin{array}{l}
    \varphi_{t}(m_{l})\;\text{ if } t< i_l\\
    \varphi_{t}(m_{l-1})\text{ otherwise};
    \end{array}\right.
    \end{array}
  $$
  $\varphi'_{t}(z)=\varphi_{t}(z)$ for all $1\le t\le h$ and $z\in\mathrsfs{F}\men\{m_0,\ldots,m_j\}$. For all $0\le i\le n$, $\mu_i=\varphi_i(m_s)$ where $0<s\le j$ satisfies $i_s\le i< i_{s+1}$, $\mu_i=\varphi_i(m_j)$ for $i_j\le i\le h$. Moreover for all $0<s\le j$, $\varphi_{i_s-1}'(m_s)\cup\{a_{i_s}\}=\varphi_{i_s-1}(m_{s-1})$ and $a_{i_1},\ldots,a_{i_s}\notin\mu_{i_s}$.
  \end{itemize}
  Let us prove this property by induction. By Lemma \ref{lem:erasing element} it is clear that if for any $z\in\mathrsfs{F}_h$ the condition $z\cup\{a_{h+1}\}=\mu_h$ does not occur, then $P_{h+1}$ is true.
  \\
  Suppose that $z\cup\{a_{h+1}\}=\mu_h$. Let us prove that there is an element $m_{j+1}\in\mathrsfs{F}$ different from $m_{l}$ for all $l\le j$ such that $z=\varphi_{h}'(m_{j+1})$. Suppose, contrary to our claim, that $m_{j+1}=m_{s}$ for some $s\le j$. We first claim that $a_{i_s}\in\varphi_{h}'(m_{s})\men\mu_{h}$. Indeed conditions $a_{i_1},\ldots,a_{i_j}\notin\mu_{i_j}$ and $s\le j$ yields to $a_{i_s}\notin\mu_{i_j}=\varphi_{i_j}(m_j)$ and so, since $\mu_{h}=\varphi_h(m_j)$ and $i_s\le i_j\le h$, we get $a_{i_s}\notin\mu_{h}$.
  Since $\varphi_{i_s-1}'(m_s)\cup\{a_{i_s}\}=\varphi_{i_s-1}(m_{s-1})$, by Lemma \ref{lem:matching property} $a_{i_s}\in\varphi_{i_s-1}(m_{s-1})=\varphi_{i_s}(m_{s-1})$, hence by property $P_h$, we get $a_{i_s}\in\varphi_{i_s-1}(m_{s-1})=\varphi_{i_s}(m_{s-1})=\varphi'_{i_s}(m_{s})$. Thus $a_{i_s}\in\varphi'_{h}(m_s)$ and so, with $a_{i_s}\notin\mu_{h}$, we get the claim $a_{i_s}\in \varphi_{h}'(m_{s})\men \mu_{h}$. However this contradicts $\varphi_{h}'(m_{s})\cup\{a_{h+1}\}=z\cup\{a_{h+1}\}=\mu_{h}$.
  \\
  Therefore we can suppose that there is a $m_{j+1}\in\mathrsfs{F}$ different from $m_0,\ldots, m_j$ such that $\varphi_{h}'(m_{j+1})\cup\{a_{h+1}\}=\mu_h$. Therefore by induction we get $\varphi_{t}'(m_{j+1})=\varphi_t(m_{j+1})$ for all $t\le h=i_{j+1}-1$. Putting $i_{j+1}=h+1$ we get, by Lemma \ref{lem:erasing element} and $P_h$
  $$
  \varphi_{i_{j+1}}'(m_{j+1})=\mu_{i_{j+1}-1}=\mu_{h}=\varphi_{h}(m_{j})=\varphi_{h+1}(m_{j})=\varphi_{i_{j+1}}(m_{j})
  $$
  since $a_{h+1}\in\mu_{h}$ and so $\mu_{h}=\varphi_{h}(m_{j})=\varphi_{h+1}(m_{j})$. Moreover we also have $\varphi_{i_{j+1}-1}'(m_{j+1})\cup\{a_{i_{j+1}}\}=\varphi_{i_{j+1}-1}(m_{j})$.
  \\
  Since $\varphi_{h}'(m_{j+1})\cup\{a_{h+1}\}=\mu_h\in \mathrsfs{F}_h$ and $\varphi_{h}'(m_{j+1})=\varphi_{h}(m_{j+1})$ we have that $\varphi_{h}(m_{j+1})=\varphi_{h+1}(m_{j+1})$ and so by Lemma \ref{lem:erasing element} the missing element becomes
  $$
  \mu_{i_{j+1}}=\varphi_{h}'(m_{j+1})=\varphi_{h}(m_{j+1})=\varphi_{h+1}(m_{j+1})=\varphi_{i_{j+1}}(m_{j+1})
  $$
  Hence $\mu_{i_{j+1}}\cup\{a_{h+1}\}=\mu_h$ and so, by Lemma \ref{lem:matching property}, $a_{h+1}\notin\mu_{i_{j+1}}$. To conclude the proof we have to show $a_{i_1},\ldots,a_{i_{j+1}}\notin\mu_{i_{j+1}}$. By induction $a_{i_1},\ldots,a_{i_j}\notin\mu_{i_j}$, hence $a_{i_1},\ldots,a_{i_j}\notin\mu_{h}$. Therefore, from $\mu_{i_{j+1}}\cup\{a_{h+1}\}=\mu_h$ and $a_{h+1}\notin\mu_{i_{j+1}}$ we get $a_{i_1},\ldots,a_{i_j},a_{i_{j+1}}\notin\mu_{i_{j+1}}$.
\end{proof}
We remark that the above swapping Lemma holds for a general family of subsets $\mathrsfs{F}$ since the hypothesis of $\cup$-closure is never used in the proof. As a consequence of the previous swapping Lemma we have the following proposition.
\begin{proposition}\label{prop:swapping proposition}
  Let $\mathrsfs{F}$ be a family of subsets of $X=\{a_1,\ldots,a_n\}$ and let $m\in \mathrsfs{F}$. Consider the rising functions $\varphi_w$, $\varphi'_w$ with respect to the word $w=a_1a_2\ldots a_n$ respectively of $\mathrsfs{F}$, $\mathrsfs{F}'=\mathrsfs{F}\setminus\{m\}$. There are $k+1$ different elements $m_i\in\mathrsfs{F}$, for $i=0,\ldots,k$ such that $m_0=m$ and for all $z\in\mathrsfs{F}\men\{m_0,\ldots,m_k\}$ we have $\varphi_w'(z)=\varphi_w(z)$ while for all $0<j\le k$
  $$
  \varphi_w'(m_j)=\varphi_w(m_{j-1})
  $$
  in particular $\varphi_w'(\mathrsfs{F}')=\varphi_w(\mathrsfs{F})\men \{\varphi_w(m_{k})\}$ and $\varphi_w(m_{k})\in\min(\varphi_w(\mathrsfs{F}))$. Moreover $a_{i_j}\in m_{j-1}\men\varphi_w(m_{j})$ for all $1\le j\le k$.
\end{proposition}
\begin{proof}
  The first claim is an immediate consequence of Lemma \ref{lem:swapping lemma} when $t=n$. In particular the missing element $\mu_{n}=\varphi_w(m_k)$ and so $\varphi_w'(\mathrsfs{F}')=\varphi_w(\mathrsfs{F})\men \{\varphi_w(m_{k})\}$. Moreover since both $\varphi_w'(\mathrsfs{F}'),\varphi_w(\mathrsfs{F})$ are upward-closed sets, then it is straightforward to prove that necessarily the missing elements must be minimal, otherwise $\varphi_w(\mathrsfs{F})\men \{\varphi_w(m_{k})\}$ would not be upward-closed, whence $\varphi_w(m_{k})\in\min(\varphi_w(\mathrsfs{F}))$.
  \\
  From Lemma \ref{lem:swapping lemma} we have that for all $1\le j\le k$, $\varphi_{i_j-1}'(m_j)\cup\{a_{i_j}\}=\varphi_{i_j-1}(m_{j-1})$ and $\varphi_{i_j-1}'(m_j)=\varphi_{i_j-1}(m_j)$, whence
  $$
  \varphi_{i_j-1}(m_j)\cup\{a_{i_j}\}=\varphi_{i_j-1}(m_{j-1})
  $$
  and so, by Lemma \ref{lem:matching property}, we get for all $1\le j\le k$, $a_{i_j}\in m_{j-1}\setminus \varphi_{w}(m_j)$.
\end{proof}
We now assume $\mathrsfs{F}$ $\cup$-closed and we consider the situation when we take away an irreducible element $m\in J(\mathrsfs{F})$. In this case we have a limitation on the number of possible swappings, indeed the following proposition holds.
\begin{proposition}\label{prop:just to cases in the irreducible case}
  Let $\mathrsfs{F}$ be a $\cup$-closed family of subsets of a set $X=\{a_1,\ldots,a_n\}$ and let $m\in J(\mathrsfs{F})$. Consider the rising functions $\varphi_w$, $\varphi'_w$ with respect to the word $w=a_1a_2\ldots a_n$ respectively of $\mathrsfs{F}$, $\mathrsfs{F}'=\mathrsfs{F}\setminus\{m\}$ and denote by $\mathcal{F} = \varphi_w(\mathrsfs{F})$, $\mathcal{F}' = \varphi'_w(\mathrsfs{F}')$. There are two possibilities: \begin{enumerate}
    \item $\varphi_w(m)\in\min(\mathcal{F})$, $\mathcal{F}'=\mathcal{F}\men \{\varphi_w(m)\}$ and for all $z\in\mathrsfs{F}'$ $\varphi_w'(z)=\varphi_w(z)$.
    \item The set $\overline{m}=\cup_{\{f\in\mathrsfs{F}: f\varsubsetneq m\}}f$ is non-empty. For all $z\in\mathrsfs{F}\men\{m, \overline{m}\}$ we have $\varphi_w'(z)=\varphi_w(z)$ and $\varphi_w'(\overline{m})=\varphi_w(m)$. Moreover $\varphi_w(\overline{m})\in\min(\mathcal{F})$ and  $\mathcal{F}'=\mathcal{F}\men \{\varphi_w(\overline{m})\}$.
  \end{enumerate}
\end{proposition}
\begin{proof}
  Using the notation of Proposition \ref{prop:swapping proposition}, suppose $k\ge 2$. Therefore, there are two distinct elements $m_1,m_2$ different from $m$ such that $\varphi_w'(m_2)=\varphi_w(m_{1})$ and $\varphi_w'(m_1)=\varphi_w(m)$. We claim
  \begin{equation}\label{eq:prop swapping with irreducible}
  \{f\in\mathrsfs{F}:f\sub \varphi_w(m_1)\}=\{f\in\mathrsfs{F}':f\sub \varphi_w(m_1)\}
  \end{equation}
  Clearly $\{f\in\mathrsfs{F}':f\sub \varphi_w(m_1)\}\sub \{f\in\mathrsfs{F}:f\sub \varphi_w(m_1)\}$ and to prove the other inclusion it is sufficient to prove that $m\nsubseteq \varphi_w(m_1)$. Suppose on the contrary that actually $m\subseteq\varphi_w(m_1)$, however by Proposition \ref{prop:swapping proposition}, $a_{i_1}\in m\men\varphi_w(m_{1})$, a contradiction. Thus (\ref{eq:prop swapping with irreducible}) holds.
  \\
  Since $m\in J(\mathrsfs{F})$, then $\mathrsfs{F}'=\mathrsfs{F}\setminus\{m\}$ is a $\cup$-closed family and so by Corollary \ref{cor:ideal correspondence}, equality (\ref{eq:prop swapping with irreducible}) and $\varphi_w'(m_2)=\varphi_w(m_{1})$ we get
  \begin{align}
  \nonumber & m_1=\varphi_w^{-1}(\varphi_w(m_1))=\bigcup_{\{f\in\mathrsfs{F}:f\sub \varphi_w(m_1)\}}f=\bigcup_{\{f\in\mathrsfs{F}':f\sub \varphi_w(m_1)\}}f\\
  \nonumber &=\bigcup_{\{f\in\mathrsfs{F}':f\sub \varphi_w'(m_2)\}}f=\varphi_w'^{-1}(\varphi_w'(m_2))=m_2
  \end{align}
  a contradiction. Therefore we have two possibilities either $k=0$ or $k=1$.
  \\
  Applying Proposition \ref{prop:swapping proposition} to the case $k=0$ we get for all $z\in\mathrsfs{F}\men\{m\}$ $\varphi_w'(z)=\varphi_w(z)$ and $\mathcal{F}'=\mathcal{F}\men \{\varphi_w(m)\}$ and $\varphi_w(m)\in\min(\mathcal{F})$.
  \\
  Consider the case $k=1$. We prove that in this case $m_1=\overline{m}$ where $\overline{m}=\cup_{\{f\in\mathrsfs{F}: f\varsubsetneq m\}}f$.
  Since $\varphi_w'(m_1)=\varphi_w(m)$ then $m_1\subsetneq\varphi_w(m)$ and so by Theorem \ref{theo:ideal correspondence} $m_1\subsetneq m$, hence $\overline{m}\neq\emptyset$. Since $\varphi_w'(m_1)=\varphi_w(m)$ then
  \begin{equation}\label{eq: 2 prop swapping with irreducible}
  \{f\in\mathrsfs{F}':f\sub \varphi'_w(m_1)\}=\{f\in\mathrsfs{F}':f\sub\varphi_w(m)\}
  \end{equation}
  moreover $\{f\in\mathrsfs{F}:f\varsubsetneq m\}\sub \{f\in\mathrsfs{F}':f\sub\varphi_w(m)\}$ and by Theorem \ref{theo:ideal correspondence} it is not difficult to check that $\{f\in\mathrsfs{F}':f\sub\varphi_w(m)\}\sub\{f\in\mathrsfs{F}:f\varsubsetneq m\}$ also holds. Hence by equality (\ref{eq: 2 prop swapping with irreducible}) we have $\{f\in\mathrsfs{F}':f\sub \varphi'_w(m_1)\}=\{f\in\mathrsfs{F}:f\varsubsetneq m\}$ and so by Corollary \ref{cor:ideal correspondence}
  $$
  m_1=\varphi_w'^{-1}(\varphi_w'(m_1))=\bigcup_{\{f\in\mathrsfs{F}':f\sub \varphi'_w(m_1)\}}f=\bigcup_{\{f\in\mathrsfs{F}:f\varsubsetneq m\}}f=\overline{m}
  $$
  The other properties are consequences of Proposition \ref{prop:swapping proposition}.
\end{proof}
We have the following theorem.
\begin{theorem}\label{theo:count irreducible}
  Let $\mathrsfs{F}$ be a $\cup$-closed family of sets of $2^X$ with $X=\{a_1,\ldots,a_n\}$. Consider the rising function $\varphi_w$ with respect to the word $w=a_1a_2\ldots a_n$ and let $\mathcal{F}=\varphi_w(\mathrsfs{F})$, then
  $$
  |J(\mathrsfs{F})|\le 2|\min(\mathcal{F})|+|\min(\mathcal{F}\men\min(\mathcal{F}))|
  $$
\end{theorem}
\begin{proof}
    $J(\mathrsfs{F})$ can be partitioned into two subsets $J_1,J_2$ respectively of the elements $m\in J(\mathrsfs{F})$ such that $\varphi_w(m)\in\min(\mathcal{F})$ and the elements $m$ for which condition 2 of Proposition \ref{prop:just to cases in the irreducible case} holds but $\varphi_w(m)\notin\min(\mathcal{F})$ (conditions 1 and 2 of Proposition \ref{prop:just to cases in the irreducible case} are not mutually exclusive). Since $\varphi_w$ is an injection we immediately have
    \begin{equation}\label{eq: J_1 less than min}
      |J_1|\le |\min(\mathcal{F})|
    \end{equation}
    We define the partial function $\iota_{\mathrsfs{F}}:J(\mathrsfs{F})\rightarrow \mathrsfs{F}$ taking an element $m$ into
    $$
    \iota_{\mathrsfs{F}}(m)=\bigcup_{f\in\mathrsfs{F}:f\subsetneq m}f
    $$
    It is straightforward to check that whenever it is defined: $\iota_{\mathrsfs{F}}(m)\subsetneq m$ (it can not be equal since $m$ is irreducible) and if $m'\subsetneq m$, then $m'\subseteq \iota_{\mathrsfs{F}}(m)$. In view of Proposition \ref{prop:just to cases in the irreducible case}, we consider the restriction $\iota_{\mathrsfs{F}}:J_2\rightarrow \mathrsfs{F}$ which is a function. Thus for a $\overline{m}\in\iota_{\mathrsfs{F}}(J_2)$, the set $\iota_{\mathrsfs{F}}^{-1}(\overline{m})$ is clearly non-empty and let $\iota_{\mathrsfs{F}}^{-1}(\overline{m})=\{m_1,\ldots, m_k\}$ for some $k\ge 1$. We observe that for all $i\neq j$, $m_i\nsubseteq m_j$ since, otherwise $m_i\subsetneq m_j$ would imply the contradiction $\overline{m}\subsetneq m_i\subseteq \iota_{\mathrsfs{F}}(m_j)=\overline{m}$. Therefore for all $i\neq j$
    \begin{equation}\label{eq:theo limitation}
    \iota_{\mathrsfs{F}}(m_j)=\iota_{\mathrsfs{F}\men\{m_i\}}(m_j)
    \end{equation}
    We claim that for all $i\neq j$ we have that at least one between $\varphi_w(m_i),\varphi_w(m_j)$ is minimal in $\mathcal{F}'=\mathcal{F}\men\{\varphi_w(\overline{m})\}$. This is a consequence of the application of Proposition \ref{prop:just to cases in the irreducible case} twice. Indeed, consider $\mathrsfs{F}\men\{m_i\}$ and let $\varphi_w'$ be the rising function of this set with respect to $w$. By Proposition \ref{prop:just to cases in the irreducible case} we have $\varphi_w(\overline{m})\in\min(\mathcal{F})$, $\varphi_w'(\overline{m})=\varphi_w(m_i)$ and $\varphi_w'(m_j)=\varphi_w(m_j)$. It is evident that $m_j\in J(\mathrsfs{F}\men\{m_i\})$ and so consider the $\cup$-closed set $(\mathrsfs{F}\men\{m_i\})\men\{m_j\}$. Let $\varphi_w''$ be the rising function of this set with respect to $w$. By Proposition \ref{prop:just to cases in the irreducible case} we have two possibilities: either
    $\varphi_w'(m_j)=\varphi_w(m_j)$ is minimal in $\mathcal{F}\men\{\varphi_w(\overline{m})\}$, or by (\ref{eq:theo limitation}), we have that
    $$
    \varphi'_w(\iota_{\mathrsfs{F}\men\{m_i\}}(m_j))=\varphi'_w(\iota_{\mathrsfs{F}}(m_j))=\varphi'_w(\overline{m})=\varphi_w(m_i)
    $$
    is minimal in $\mathcal{F}\men\{\varphi_w(\overline{m})\}$. Therefore, it is straightforward to prove that all the $m_i$ except at most one, say $m_k$, are minimal in $\mathcal{F}\men\{\varphi_w(\overline{m})\}$. Hence, denoting by $J_2'$ the set of elements $m\in J_2$ such that $\varphi_w(m)$ is minimal in $\mathcal{F}\men\{\varphi_w(\iota_{\mathrsfs{F}}(m))\}$, we get that there is an injection of $J_2\men J_2'$ into $\iota_{\mathrsfs{F}}(J_2\men J_2')$ which is in one to one correspondence with the elements of $\varphi_w(\iota_{\mathrsfs{F}}(J_2\men J_2'))$ (being $\varphi_w$ injective) which is in turn a subset of $\min(\mathcal{F})$ (by definition of the set $J_2$ and Proposition \ref{prop:just to cases in the irreducible case}), whence:
    \begin{equation}\label{eq: J_2 minus J_2'}
      |J_2\men J_2'|\le|\min(\mathcal{F})|
    \end{equation}
    We now prove that $\varphi_w(J_2')\subseteq \min(\mathcal{F}\men\min(\mathcal{F}))$. Since $J_2'\subseteq J_2$, then, by definition of $J_2$, we have that $\varphi_w(m)\notin \min(\mathcal{F})$ for all $m\in J_2'$. Thus $\varphi_w(J_2')\subseteq \mathcal{F}\men\min(\mathcal{F})$. Moreover, if $m\in J_2'$, then $\varphi_w(m)$ is minimal in $\mathcal{F}\men\{\varphi_w(\iota_{\mathrsfs{F}}(m))\}$ and since $\varphi_w(\iota_{\mathrsfs{F}}(m))\in min(\mathcal{F})$ we have
    $$
    \mathcal{F}\men\min(\mathcal{F})\subseteq\mathcal{F}\men\{\varphi_w(\iota_{\mathrsfs{F}}(m))\}
    $$
    hence $\varphi_w(m)$ is also minimal in $\mathcal{F}\men\min(\mathcal{F})$, and so the claim $\varphi_w(J_2')\subseteq \min(\mathcal{F}\men\min(\mathcal{F}))$. Therefore $|J_2'|\le |\min(\mathcal{F}\men\min(\mathcal{F}))|$, and so by (\ref{eq: J_1 less than min}), (\ref{eq: J_2 minus J_2'}) we obtain the upper bound of the statement
    $$
    |J(\mathrsfs{F})|=|J_1|+|J_2\men J_2'|+|J_2'|\le 2|\min(\mathcal{F})|+|\min(\mathcal{F}\men\min(\mathcal{F}))|
    $$
\end{proof}
As an immediate consequence of the previous theorem and Sperner's Theorem we have $|J(\mathrsfs{F})|\le 3{n\choose \lfloor\frac{n}{2}\rfloor}$. This bound is not the best that can be obtained from Theorem \ref{theo:count irreducible}. Indeed, we devote Subsection \ref{subsec: extremal problem} to prove Theorem \ref{theo: antichain + nabla bound} showing that for an upward-closed family $\mathcal{F}$ on a set $X$ with $|X|=n$ we have $2|\min(\mathcal{F})|+|\min(\mathcal{F}\men\min(\mathcal{F}))|\le 2{n\choose \lfloor\frac{n}{2}\rfloor}+{n\choose \lfloor\frac{n}{2}\rfloor+1}$ and this bound is tight. Therefore we have the following corollary.
\begin{corollary}\label{cor:bound irreducible}
  Let $\mathrsfs{F}$ be a $\cup$-closed family of sets of $2^X$ with $X=\{a_1,\ldots,a_n\}$, then
  $$
  |J(\mathrsfs{F})|\le 2{n\choose \lfloor\frac{n}{2}\rfloor}+{n\choose \lfloor\frac{n}{2}\rfloor+1}
  $$
  In particular any family $S\sub 2^X$ with $|S|>2{n\choose \lfloor\frac{n}{2}\rfloor}+{n\choose \lfloor\frac{n}{2}\rfloor+1}$ is not $\cup$-independent.
\end{corollary}
A natural question that arises from this corollary is the precise upper bound of the quantity
$$
J(n) = \max\{|J(\mathrsfs{F})|: \mathrsfs{F}\mbox{ is a }\cup-\mbox{closed family on a set }X\mbox{ with } |X| = n\}
$$
Although we are not able to answer to this question we can easily give a lower bound to $J(n)$. Indeed, consider the $\cup$-closed family of $2^X$ consisting of elements whose cardinality is greater than or equal to $\lfloor\frac{n}{2}\rfloor$. The set of joint-irreducible elements consists of the subsets of cardinality exactly $\lfloor\frac{n}{2}\rfloor$, whence we can bound the function $J(n)$ as
$$
{n\choose \lfloor\frac{n}{2}\rfloor}\le J(n) \le 2{n\choose
\lfloor\frac{n}{2}\rfloor}+{n\choose
\lfloor\frac{n}{2}\rfloor+1}
$$

\subsection{An extremal problem}\label{subsec: extremal problem}
In this section we study the extremal problem of maximizing the quantity $2|\min(\mathcal{F})|+|\min(\mathcal{F}\men\min(\mathcal{F}))|$ where $\mathcal{F}$ is an upward-closed set on the set $X$. We can restate this problem in the following way. Given an antichain $\mathcal{A}$ of $2^X$, we want to maximize the quantity $2|\mathcal{A}|+|\min(\mathcal{A}^\uparrow\setminus \mathcal{A})|$. Before studying this problem more in detail we give some definitions. For an integer $0 <k\le n$ we denote by $\mathcal{A}_k =\{A\in\mathcal{A}:|A| = k\}$, in general a family of $k$-subsets $\mathcal{B}$ is a collection of sets of $X$ with cardinality $k$. We recall that the \emph{shade} (see \cite{Anderson}) of $\mathcal{A}_k$ is defined by
$$
\nabla(\mathcal{A}_k) = \{B\in 2^X: |B| = k+1, A\subseteq B\mbox{ for some }A\in\mathcal{A}_k\}
$$
Similarly the \emph{shadow} of $\mathcal{A}_k$ is defined by
$$
\Delta(\mathcal{A}_k) = \{B\in 2^X: |B| = k-1, B\subseteq A\mbox{ for some }A\in\mathcal{A}_k\}
$$
We can extend these definitions to the whole set $\mathcal{A}$ by taking $\nabla(\mathcal{A})= \cup_{k = 1}^n\nabla(\mathcal{A}_k)$ and $\Delta(\mathcal{A})=\cup_{k = 1}^n\Delta(\mathcal{A}_k)$. Note that $\min(\mathcal{A}^\uparrow\setminus \mathcal{A})\subseteq \nabla(\mathcal{A})$, in particular, since $\mathcal{A}$ is an antichain, $\min(\mathcal{A}^\uparrow\setminus \mathcal{A}) =  \min(\nabla(\mathcal{A}))$. Thus, it makes sense defining the \emph{first upward level} of an antichain $\mathcal{A}$ as the set $\overline{\nabla}\mathcal{A} = \min(\nabla(\mathcal{A}))$. The operator $\overline{\nabla}$ is also interesting because $\mathcal{A}^{\uparrow}$ can be partitioned into ``foils", where the \emph{$i$-th foil} for $i\ge 1$ is given by $\overline{\nabla}^i(\mathcal{A}) = \overline{\nabla}(\overline{\nabla}^{i-1}(\mathcal{A}))$ and $\overline{\nabla}^0(\mathcal{A}) = \mathcal{A}$. We state some useful properties whose proofs are left to the reader.
\begin{lemma}\label{lem: properties overline nabla}
  Let $\mathcal{A},\mathcal{B}$ be two antichains, then:
  \begin{enumerate}
    \item\label{nabla prop 6} $\nabla(\mathcal{A}\cup\mathcal{B})= \nabla(\mathcal{A})\cup\nabla(\mathcal{B})$, $\Delta(\mathcal{A}\cup\mathcal{B})= \Delta(\mathcal{A})\cup\Delta(\mathcal{B})$, $\mathcal{A}\subseteq\nabla(\Delta(\mathcal{A})), \mathcal{A}\subseteq\Delta(\nabla(\mathcal{A}))$.
    \item\label{nabla prop 7} Assume $\mathcal{A}\subseteq \mathcal{B}$, then for any $g\in\nabla(\mathcal{A})$ there is a $g'\in\overline{\nabla}(\mathcal{B})$ such that $g'\subseteq g$.
    \item\label{nabla prop 5} $\overline{\nabla}(\mathcal{A})\subseteq \nabla(\mathcal{A})$, moreover $g\in \nabla(\mathcal{A})\setminus \overline{\nabla}(\mathcal{A})$ iff there is $g'\in \nabla(\mathcal{A})$ such that $g'\subsetneq g$.
    \item\label{nabla prop 1} If $\mathcal{A}\cup\mathcal{B}$ is an antichain and for all $g\in\overline{\nabla}(\mathcal{A})$ there is no $g'\in \overline{\nabla}(\mathcal{B})$ such that $g'\subsetneq g$, then $\overline{\nabla}(\mathcal{A})\subseteq\overline{\nabla}(\mathcal{A}\cup\mathcal{B})$.
    \item\label{nabla prop 4} Assume $\mathcal{A}\subseteq \mathcal{B}$, if for any $g\in\nabla(\mathcal{A})$ there is a $g'\in \nabla(\mathcal{B}\setminus\mathcal{A})$ such that $g'\subseteq g$, then $\overline{\nabla}(\mathcal{B}\setminus\mathcal{A})=\overline{\nabla}(\mathcal{B})$.
  \end{enumerate}
\end{lemma}
We devote the rest of the paper to the proof of following theorem.
\begin{theorem}\label{theo: antichain + nabla bound}
    Let $\mathcal{A}$ be an antichain of $2^X$ with $|X| = n$, then
    $$
    2|\mathcal{A}|+|\overline{\nabla}(\mathcal{A})|\le 2{n\choose \lfloor\frac{n}{2}\rfloor} + {n\choose \lfloor\frac{n}{2}\rfloor + 1}
    $$
    and this bound is tight.
\end{theorem}
Note that if $n$ is odd the theorem can be easily proved. Indeed, both $|\mathcal{A}|$ and $|\overline{\nabla}(\mathcal{A})|$ are antichains, hence by Sperner's theorem we get $2|\mathcal{A}|+|\overline{\nabla}(\mathcal{A})|\le 3{n\choose \lfloor\frac{n}{2}\rfloor}=2{n\choose \lfloor\frac{n}{2}\rfloor} + {n\choose \lfloor\frac{n}{2}\rfloor + 1}$ since ${n\choose \lfloor\frac{n}{2}\rfloor} = {n\choose \lfloor\frac{n}{2}\rfloor + 1}$. It is not difficult to check that this bound is attained when $\mathcal{A}$ consists of all the $\frac{n-1}{2}$-subsets of $X$. Therefore, in the sequel we can assume that $n$ is even. We prove the theorem using an augmentation argument. More precisely, we define two maps $\alpha^{+}, \alpha^{-}$, called respectively the \emph{upward-augmenting}, \emph{lower-augmenting} map, with the property of transforming $\mathcal{A}$ into the antichains $\alpha^{+}(\mathcal{A})$, $\alpha^{-}(\mathcal{A})$ with
\begin{eqnarray*}
  \nonumber 2|\mathcal{A}|+|\overline{\nabla}(\mathcal{A})| &\le & 2|\alpha^{+}(\mathcal{A})|+|\overline{\nabla}(\alpha^{+}(\mathcal{A}))| \\
  \nonumber 2|\mathcal{A}|+|\overline{\nabla}(\mathcal{A})| &\le & 2|\alpha^{-}(\mathcal{A})|+|\overline{\nabla}(\alpha^{-}(\mathcal{A}))|
\end{eqnarray*}
then we repetitively apply these operators to obtain an antichain formed by $k$-subsets of $X$ with $k=\frac{n}{2},\frac{n}{2}-1$.
However we define these maps only for particular classes of antichains that we are going to introduce, first we need some preliminary definitions. Given a family $\mathcal{B}\subseteq 2^X$ we denote the maximum (minimum) of the lengths of the elements of $\mathcal{B}$ by $\parallel \mathcal{B}\parallel_M $ ($\parallel \mathcal{B}\parallel_m $), and we put $Max(\mathcal{B}) = \{B\in\mathcal{B}: |B|= \parallel \mathcal{B}\parallel_M\}$, $Min(\mathcal{B}) = \{B\in\mathcal{B}: |B|= \parallel \mathcal{B}\parallel_m\}$.
The following lemma shows that we can restrict our attention to a particular class of antichains.
\begin{lemma}\label{lem:augmentable antichains}
  Let $\mathcal{A}'$ be an antichain in $2^X$, then
  \begin{enumerate}
    \item[1)] for any $h\in Min(\mathcal{A}')$ and $a\in X\setminus h$ we have $h\cup\{a\}\in\overline{\nabla}(\mathcal{A}')$.
    \end{enumerate}
  Moreover there is an antichain $\mathcal{A}\supseteq\mathcal{A}'$ such that $\parallel \mathcal{A}\parallel_M = \parallel \mathcal{A}'\parallel_M$, $\parallel \mathcal{A}\parallel_m = \parallel \mathcal{A}'\parallel_m$, $2|\mathcal{A}'|+|\overline{\nabla}(\mathcal{A}')|\le 2|\mathcal{A}|+|\overline{\nabla}(\mathcal{A})|$ and with the following property:
  \begin{enumerate}
    \item[2)]let $k = \parallel \overline{\nabla}(\mathcal{A})\parallel_M$, then either $\cup_{i\ge k}\mathcal{A}_i\neq\emptyset$ or for any $h\in Max(\overline{\nabla}(\mathcal{A}))$ and $a\in h$ we have $h\setminus\{a\}\in\mathcal{A}$.
  \end{enumerate}
\end{lemma}
\begin{proof}
    Let us prove Condition 1). To obtain a contradiction suppose that there is $h\in Min(\mathcal{A}')$ and $a\in X\setminus h$ such that $h\cup\{a\}\notin\overline{\nabla}(\mathcal{A})$. Thus $h\cup\{a\}\in\nabla(\mathcal{A}')\setminus \overline{\nabla}(\mathcal{A}')$, and so by property \ref{nabla prop 5} of Lemma \ref{lem: properties overline nabla}, there is a $g\in \nabla(\mathcal{A}')$ and $h'\in\mathcal{A}'$ with $h'\subsetneq g\subsetneq h\cup\{a\}$, whence $|h'|<|g|<|h|+1$. Thus $|h'|<|h|$ which contradicts the minimality of $|h|$.
    \\
    The second statement is proved if we show that given an antichain $\mathcal{B}'$ with $k = \parallel \overline{\nabla}(\mathcal{B}')\parallel_M$ either $\cup_{i\ge k}\mathcal{B}'_i\neq\emptyset$ or if there is an $h\in Max(\overline{\nabla}(\mathcal{B}'))$ and $a\in h$ such that $h\setminus\{a\}\notin\mathcal{B}'$, then the family $\mathcal{B}=\mathcal{B}'\cup\{h\setminus\{a\}\}$ is an antichain with $2|\mathcal{B}'|+|\overline{\nabla}(\mathcal{B}')|\le 2|\mathcal{B}|+|\overline{\nabla}(\mathcal{B})|$. Indeed starting from $\mathcal{A}'$ by repetitively adding elements for which condition 2) does not hold, we eventually end with an antichain $\mathcal{A}$ satisfying property 2). Suppose that $\cup_{i\ge k}\mathcal{B}'_i = \emptyset$, otherwise we have done. It is easily seen that $\parallel \mathcal{B}\parallel_M = \parallel \mathcal{B}'\parallel_M$, $\parallel \mathcal{B}\parallel_m = \parallel \mathcal{B}'\parallel_m$.
    We now prove that $\mathcal{B}$ is an antichain. Note first that $h$ can not be a singleton, thus to reach a contradiction suppose that $\mathcal{B}$ is not an antichain. Since $\mathcal{B}'$ is an antichain, there is a $g\in\mathcal{B}'$ such that either $g\subsetneq h\setminus\{a\}$ or $h\setminus\{a\} \subsetneq g $. Suppose that $g\subsetneq h\setminus\{a\}$, hence there is a $h'\in \overline{\nabla}(\mathcal{B}')$ such that $h'\subseteq g\cup\{a\}\subsetneq h$ which contradicts the fact that $\overline{\nabla}(\mathcal{B}')$ is an antichain. On the other hand suppose that $h\setminus\{a\} \subsetneq g $. Thus, $|g|\ge |h| = k$ which implies $g\in\cup_{i\ge k}\mathcal{B}'_i=\emptyset$, a contradiction. Therefore $\mathcal{B} = \mathcal{B}'\cup\{h\setminus\{a\}\}$ is an antichain. We now prove that $\overline{\nabla}(\mathcal{B}')\subseteq \overline{\nabla}(\mathcal{B})$ from which, with $\mathcal{B} = \mathcal{B}'\cup\{h\setminus\{a\}\}$, implies our claim $2|\mathcal{B}'|+|\overline{\nabla}(\mathcal{B}')|\le 2|\mathcal{B}|+|\overline{\nabla}(\mathcal{B})|$. Suppose, contrary to our claim, that there is a $t\in \overline{\nabla}(\mathcal{B}')\setminus \overline{\nabla}(\mathcal{B}) \neq\emptyset$. It is straightforward to check that there is a $t'\in \overline{\nabla}(\mathcal{B})$ with $t'\subsetneq t$. It follows easily that $h\setminus\{a\}\subsetneq t'\subsetneq t$ (otherwise we would have the contradiction $t'\in \overline{\nabla}(\mathcal{B}')$). Thus we have $|t|>|t'|\ge |h|=k$, against $\parallel\overline{\nabla}(\mathcal{B}')\parallel_M = k$.
\end{proof}
An antichain $\mathcal{A}$ satisfying the two properties in Lemma \ref{lem:augmentable antichains} is called \emph{augmentable}. We now define the lower-augmenting, upward-augmenting map on the set of augmentable antichains over $X$. Given an augmentable antichain $\mathcal{A}$ with $k = \parallel \overline{\nabla}(\mathcal{A})\parallel_M, k'=\parallel \mathcal{A}\parallel_M$, $s=\parallel \mathcal{A}\parallel_m$, the lower-augmenting map is defined by
$$
\alpha^{-}(\mathcal{A}) = \left\{
  \begin{array}{ll}
    (\mathcal{A}\setminus\mathcal{A}_{k'})\cup \Delta(\mathcal{A}_{k'}), & \hbox{if } k'\ge k,\: k'\ge \frac{n}{2}+1 \\
    (\mathcal{A}\setminus\mathcal{A}_{k-1}) \cup \Delta(\mathcal{A}_{k-1}), & \hbox{if } k'<k,\: k > \frac{n}{2}+1\\
    \mathcal{A} & \hbox{otherwise.}
  \end{array}
\right.
$$
and the upward-augmenting map by
$$
\alpha^{+}(\mathcal{A}) = \left\{
  \begin{array}{ll}
    (\mathcal{A}\setminus \mathcal{A}_s)\cup \nabla(\mathcal{A}_s), & \hbox{if } s < \frac{n}{2}-1 \\
    \mathcal{A} & \hbox{otherwise.}
  \end{array}
\right.
$$
The following lemma shows that $\alpha^+(\mathcal{A}),\alpha^-(\mathcal{A})$ are antichains.
\begin{lemma}\label{lem: well defined}
  Let $\mathcal{A}$ be an antichain and let $M=\parallel \mathcal{A}\parallel_M$, $m=\parallel \mathcal{A}\parallel_m$, then
  $$
  \mathcal{A}\setminus \mathcal{A}_m\cup\nabla(\mathcal{A}_m), \: \mathcal{A}\setminus \mathcal{A}_M\cup\Delta(\mathcal{A}_M)
  $$
  are antichains with $\mathcal{A}\setminus \mathcal{A}_m\cap\nabla(\mathcal{A}_m)=\emptyset$, $\mathcal{A}\setminus \mathcal{A}_M\cap\Delta(\mathcal{A}_M)=\emptyset$.
\end{lemma}
\begin{proof}
  Suppose, contrary to our claim, that there is $g\in\mathcal{A}\setminus \mathcal{A}_m\cap\nabla(\mathcal{A}_m) \neq\emptyset$. Thus $g\in\nabla(\mathcal{A}_m)$ implies that there is a $g'\in\mathcal{A}_m$ such that $g'\subsetneq g$ which contradicts the fact that $\mathcal{A}$ is an antichain. Similarly, $\mathcal{A}_M\cap\Delta(\mathcal{A}_M)\neq\emptyset$ contradicts the fact that $\mathcal{A}$ is an antichain. Let us prove that $\mathcal{A}\setminus \mathcal{A}_m\cup\nabla(\mathcal{A}_m)$ is an antichain. Since the two terms of the union are disjoint antichains, to reach a contradiction, we can suppose that there is a $g\in \mathcal{A}\setminus \mathcal{A}_m$ and $g'\in \nabla(\mathcal{A}_m)$ such that either $g\subsetneq g'$ or $g'\subsetneq g$. Since $m$ is the minimum of the length of the elements of $\mathcal{A}$, then $g\in \mathcal{A}\setminus \mathcal{A}_m$ implies $|g|\ge m+1$, while $g'\in \nabla(\mathcal{A}_m)$ implies $|g'|=m+1$. Thus only $g'\subsetneq g$ can occur. However $g'\in \nabla(\mathcal{A}_m)$ implies $g''\subsetneq g'$, for some $g''\in \mathcal{A}_m$. Hence $g''\subsetneq g'\subsetneq g$ which contradicts the fact that $\mathcal{A}$ is an antichain. Hence $\mathcal{A}\setminus \mathcal{A}_m\cup\nabla(\mathcal{A}_m)$ is an antichain. Let us prove that $\mathcal{A}\setminus \mathcal{A}_M\cup\Delta(\mathcal{A}_M)$ is also an antichain. Suppose, contrary to our claim, that $\mathcal{A}\setminus \mathcal{A}_M\cup\Delta(\mathcal{A}_M)$ is not an antichain.
  Similarly to the above situation, we can assume that only $g\subsetneq g'$ for $g\in \mathcal{A}\setminus \mathcal{A}_M$ and $g'\in \Delta(\mathcal{A}_M)$ can occur. However, $g'\in \Delta(\mathcal{A}_M)$ implies that there is a $g''\in \mathcal{A}_M$ with $g'\subsetneq g''$, hence $g\subsetneq g'\subsetneq g''$ contradicts the fact that $\mathcal{A}$ is an antichain. Therefore $\mathcal{A}\setminus \mathcal{A}_M\cup\Delta(\mathcal{A}_M)$ is an antichain and this concludes the proof of the lemma.
\end{proof}

\begin{lemma}\label{lem: alpha +}
  Let $\mathcal{A}$ be an augmentable antichain, then $\alpha^+(\mathcal{A})$ is antichain with $\parallel \alpha^+(\mathcal{A})\parallel_m> \parallel \mathcal{A}\parallel_m$, if $\parallel \mathcal{A}\parallel_m<\parallel \mathcal{A}\parallel_M$ then $\parallel \alpha^+(\mathcal{A})\parallel_M = \parallel \mathcal{A}\parallel_M$, moreover:
   $$
   2|\mathcal{A}|+|\overline{\nabla}(\mathcal{A})| \le 2|\alpha^{+}(\mathcal{A})|+|\overline{\nabla}(\alpha^{+}(\mathcal{A}))|
   $$
\end{lemma}
\begin{proof}
  Let $s=\parallel \mathcal{A}\parallel_m$, it is evident that $\alpha^+$ substitutes $Min(\mathcal{A})$ with $\nabla(Min(\mathcal{A}))$. Thus $\parallel \alpha^+(\mathcal{A})\parallel_m  > \parallel \mathcal{A}\parallel_m$. Moreover if $s<\parallel \mathcal{A}\parallel_M$ then, since we just add elements of cardinality $s+1$, it is also immediate that $\parallel \alpha^+(\mathcal{A})\parallel_M = \parallel \mathcal{A}\parallel_M$.
  \\
  By Lemma \ref{lem: well defined} $\alpha^{+}(\mathcal{A})$ is an antichain with:
  \begin{equation}\label{eq: disjoint antichain con nabla}
    \mathcal{A}\setminus \mathcal{A}_s\cap\nabla(\mathcal{A}_s) =\emptyset
  \end{equation}
  Let us prove the inequality $2|\mathcal{A}|+|\overline{\nabla}(\mathcal{A})| \le 2|\alpha^{+}(\mathcal{A})|+|\overline{\nabla}(\alpha^{+}(\mathcal{A}))|$. We first claim that
  \begin{equation}\label{eq: decomposition nabla}
  \overline{\nabla}(\alpha^{+}(\mathcal{A})) \supseteq \overline{\nabla}(\mathcal{A})\setminus \nabla(\mathcal{A}_s)\cup \nabla(\nabla(\mathcal{A}_s))
  \end{equation}
  where $ \nabla(\mathcal{A}_s)\subseteq \overline{\nabla}(\mathcal{A})$ and
  \begin{equation}\label{eq: nabla intersection empty}
    \overline{\nabla}(\mathcal{A})\setminus \nabla(\mathcal{A}_s)\cap \nabla(\nabla(\mathcal{A}_s)) = \emptyset
  \end{equation}
  By property 1) of an augmentable chain $\mathcal{A}$ we have $ \nabla(\mathcal{A}_s)\subseteq \overline{\nabla}(\mathcal{A})$. Let us prove (\ref{eq: nabla intersection empty}). Suppose that (\ref{eq: nabla intersection empty}) do not hold and let $h\in \overline{\nabla}(\mathcal{A})\setminus \nabla(\mathcal{A}_s)\cap \nabla(\nabla(\mathcal{A}_s))$. Thus $h = g\cup\{a,b\}$ for some $g\in \mathcal{A}_s$ and $a,b\notin g$, since $g' = g\cup\{a\}\in\overline{\nabla}(\mathcal{A})$ we have $g'\subsetneq h$ for $g',h\in \overline{\nabla}(\mathcal{A})$, a contradiction. Let us prove (\ref{eq: decomposition nabla}). We split the proof of (\ref{eq: decomposition nabla}) by showing first $\overline{\nabla}(\mathcal{A})\setminus \nabla(\mathcal{A}_s)\subseteq \overline{\nabla}(\mathcal{A}\setminus \mathcal{A}_s\cup \nabla(\mathcal{A}_s))$ and then $\nabla(\nabla(\mathcal{A}_s))\subseteq \overline{\nabla}(\mathcal{A}\setminus \mathcal{A}_s\cup \nabla(\mathcal{A}_s))$.
  \\
  \begin{itemize}
    \item Case $\overline{\nabla}(\mathcal{A})\setminus \nabla(\mathcal{A}_s)\subseteq \overline{\nabla}(\mathcal{A}\setminus \mathcal{A}_s\cup \nabla(\mathcal{A}_s))$.   Since $\overline{\nabla}(\mathcal{A})_{s+1}\subseteq \nabla(\mathcal{A}_s)\subseteq\overline{\nabla}(\mathcal{A})_{s+1}$, then $\overline{\nabla}(\mathcal{A})_{s+1}= \nabla(\mathcal{A}_s)$. Thus $\overline{\nabla}(\mathcal{A})\setminus \nabla(\mathcal{A}_s)\subseteq \nabla(\mathcal{A}\setminus\mathcal{A}_s)$, and so, by property \ref{nabla prop 6}) of Lemma \ref{lem: properties overline nabla}, we get $\overline{\nabla}(\mathcal{A})\setminus \nabla(\mathcal{A}_s)\subseteq\nabla(\mathcal{A}\setminus\mathcal{A}_s\cup\nabla(\mathcal{A}_s))$. Suppose, contrary to our claim, that there is a $g\in\overline{\nabla}(\mathcal{A})\setminus \nabla(\mathcal{A}_s) $ such that $g\in \nabla(\mathcal{A}\setminus\mathcal{A}_s\cup\nabla(\mathcal{A}_s))\setminus\overline{\nabla}(\mathcal{A}\setminus\mathcal{A}_s\cup\nabla(\mathcal{A}_s))$. Therefore, by properties \ref{nabla prop 5}), \ref{nabla prop 6}) of Lemma \ref{lem: properties overline nabla} there is a $g'\in \nabla(\mathcal{A}\setminus\mathcal{A}_s\cup\nabla(\mathcal{A}_s))=\nabla(\mathcal{A}\setminus\mathcal{A}_s)\cup\nabla(\nabla(\mathcal{A}_s))$ such that $g'\subsetneq g$. We consider two cases, either $g'\in \nabla(\mathcal{A}\setminus\mathcal{A}_s)$ or $g'\in\nabla(\nabla(\mathcal{A}_s))$. Suppose that $g'\in \nabla(\mathcal{A}\setminus\mathcal{A}_s)$, then by property \ref{nabla prop 7}) of Lemma \ref{lem: properties overline nabla}, there is a $g''\in \overline{\nabla}(\mathcal{A})$ such that $g''\subseteq g'\subsetneq g\in \overline{\nabla}(\mathcal{A})$ which contradicts the fact that $\overline{\nabla}(\mathcal{A})$ is an antichain. On the other hand, suppose that $g'\in\nabla(\nabla(\mathcal{A}_s))$. Hence there is a $g''\in \nabla(\mathcal{A}_s)\subseteq\overline{\nabla}(\mathcal{A})$ such that $g''\subsetneq g'\subsetneq g\in \overline{\nabla}(\mathcal{A})$ which again contradicts the fact that $\overline{\nabla}(\mathcal{A})$ is an antichain. Hence we conclude $\overline{\nabla}(\mathcal{A})\setminus \nabla(\mathcal{A}_s)\subseteq \overline{\nabla}(\mathcal{A}\setminus \mathcal{A}_s\cup \nabla(\mathcal{A}_s))$.
    \item Case $\nabla(\nabla(\mathcal{A}_s))\subseteq \overline{\nabla}(\mathcal{A}\setminus \mathcal{A}_s\cup \nabla(\mathcal{A}_s))$. It is evident by property \ref{nabla prop 7}) of Lemma \ref{lem: properties overline nabla} that $\nabla(\nabla(\mathcal{A}_s))\subseteq \nabla(\mathcal{A}\setminus \mathcal{A}_s\cup \nabla(\mathcal{A}_s))$. Suppose, contrary to our claim, that there is a $g\in \nabla(\nabla(\mathcal{A}_s))$ such that $g\in \nabla(\mathcal{A}\setminus \mathcal{A}_s\cup \nabla(\mathcal{A}_s))\setminus\overline{\nabla}(\mathcal{A}\setminus \mathcal{A}_s\cup \nabla(\mathcal{A}_s))$. Therefore, by properties \ref{nabla prop 5}), \ref{nabla prop 6}) of Lemma \ref{lem: properties overline nabla} there is a $g'\in \nabla(\mathcal{A}\setminus\mathcal{A}_s\cup\nabla(\mathcal{A}_s))=\nabla(\mathcal{A}\setminus\mathcal{A}_s)\cup\nabla(\nabla(\mathcal{A}_s))$ such that $g'\subsetneq g$. Also in this case we consider the two cases either $g'\in \nabla(\mathcal{A}\setminus\mathcal{A}_s)$ or $g'\in\nabla(\nabla(\mathcal{A}_s))$. Suppose that $g'\in \nabla(\mathcal{A}\setminus\mathcal{A}_s)$. Since $s=\parallel \mathcal{A}\parallel_m$, then $|g'|\ge s+2$, while $g\in \nabla(\nabla(\mathcal{A}_s))$ implies $|g|=s+2$ which contradicts $g'\subsetneq g$. In the other case, if $g'\in \nabla(\nabla(\mathcal{A}_s))$, then $g\in \nabla(\nabla(\mathcal{A}_s))$ and $g'\subsetneq g$ contradict the fact that $\nabla(\nabla(\mathcal{A}_s))$ is an antichain. Hence $\nabla(\nabla(\mathcal{A}_s))\subseteq \overline{\nabla}(\mathcal{A}\setminus \mathcal{A}_s\cup \nabla(\mathcal{A}_s))$ and this completes the proof of (\ref{eq: decomposition nabla}).
  \end{itemize}
  Let us complete the proof of the lemma showing the inequality in the statement. Since $s<\frac{n}{2}-1$, then by \cite[Corollary 2.1.2]{Anderson}, $|\nabla(\mathcal{A}_s)|-|\mathcal{A}_s|\ge 0$ and $|\nabla(\nabla(\mathcal{A}_s))|-|\nabla(\mathcal{A}_s)|\ge 0$. By (\ref{eq: disjoint antichain con nabla}), (\ref{eq: decomposition nabla}), (\ref{eq: nabla intersection empty}) we have $ 2|\alpha^{+}(\mathcal{A})|+|\overline{\nabla}(\alpha^{+}(\mathcal{A}))| \ge 2|\mathcal{A}| - 2|\mathcal{A}_s| + 2|\nabla(\mathcal{A}_s)|+ |\overline{\nabla}(\mathcal{A})\setminus \nabla(\mathcal{A}_s)| + |\nabla(\nabla(\mathcal{A}_s))|$. Furthermore, using $\nabla(\mathcal{A}_s)\subseteq \overline{\nabla}(\mathcal{A})$, $|\nabla(\mathcal{A}_s)|-|\mathcal{A}_s|\ge 0$ and $|\nabla(\nabla(\mathcal{A}_s))|-|\nabla(\mathcal{A}_s)|\ge 0$ we get $2|\alpha^{+}(\mathcal{A})|+|\overline{\nabla}(\alpha^{+}(\mathcal{A}))| \ge 2|\mathcal{A}| +  |\overline{\nabla}(\mathcal{A})|$.
 \end{proof}
\begin{lemma}\label{lem: alpha -}
    Let $\mathcal{A}$ be an augmentable antichain, then $\alpha^-(\mathcal{A})$ is an antichain with $\parallel \alpha^-(\mathcal{A})\parallel_M< \parallel \mathcal{A}\parallel_M$, if $\parallel \mathcal{A}\parallel_m<\parallel \mathcal{A}\parallel_M$ then $\parallel \alpha^-(\mathcal{A})\parallel_m = \parallel \mathcal{A}\parallel_m$, moreover:
    $$
   2|\mathcal{A}|+|\overline{\nabla}(\mathcal{A})| \le 2|\alpha^{-}(\mathcal{A})|+|\overline{\nabla}(\alpha^{-}(\mathcal{A}))|
    $$
\end{lemma}
\begin{proof}
  Let $k = \parallel \overline{\nabla}(\mathcal{A})\parallel_M, k'=\parallel \mathcal{A}\parallel_M$. For this operator we need to consider two cases: either $k'\ge k$ and $k'\ge\frac{n}{2}+1$, or $k'< k$ and $k>\frac{n}{2}+1$. Note that in the case $k'< k$, since $\mathcal{A}$ is augmentable, then by property 2) of Lemma \ref{lem:augmentable antichains} we have $k' = k-1$. In both cases the map $\alpha^-$ substitutes $ Max(\mathcal{A})$ with $\Delta( Max(\mathcal{A}))$, thus $\parallel \alpha^-(\mathcal{A})\parallel_M< \parallel \mathcal{A}\parallel_M$ holds and if $\parallel \mathcal{A}\parallel_m<\parallel \mathcal{A}\parallel_M$ then it is also obvious that $\parallel \alpha^-(\mathcal{A})\parallel_m = \parallel \mathcal{A}\parallel_m$.
  \\
  Consider now the case $k'\ge k$. By Lemma \ref{lem: well defined}, $\alpha^-(\mathcal{A}) = (\mathcal{A}\setminus\mathcal{A}_{k'})\cup \Delta(\mathcal{A}_{k'})$ is an antichain with:
  \begin{equation}\label{eq: disjoint antichain con nabla case - 1}
  (\mathcal{A}\setminus\mathcal{A}_{k'})\cap \Delta(\mathcal{A}_{k'})=\emptyset
  \end{equation}
  We now prove the inequality of the statement. We claim
  \begin{equation}\label{eq: nabla dec of alpha - first case}
  \overline{\nabla}(\alpha^-(\mathcal{A})) \supseteq \overline{\nabla}(\mathcal{A})
  \end{equation}
  We first prove that $\overline{\nabla}(\mathcal{A}\setminus\mathcal{A}_{k'})=\overline{\nabla}(\mathcal{A})$. Since $k = \parallel \overline{\nabla}(\mathcal{A})\parallel_M$ and $k'\ge k$, then any element in $\nabla(\mathcal{A}_{k'})$ contains some element in $\nabla(\mathcal{A}\setminus\mathcal{A}_{k'})$. Thus by property \ref{nabla prop 4}) of Lemma \ref{lem: properties overline nabla} we have the claim $\overline{\nabla}(\mathcal{A}\setminus\mathcal{A}_{k'})=\overline{\nabla}(\mathcal{A})$. If we show that for any $g\in\overline{\nabla}(\mathcal{A})$ there is no $g'\in\overline{\nabla}(\Delta(\mathcal{A}_{k'}))$ such that $g'\subsetneq g$, then the inclusion (\ref{eq: nabla dec of alpha - first case}) follows from property 4) of Lemma \ref{lem: properties overline nabla} and $\overline{\nabla}(\mathcal{A}\setminus\mathcal{A}_{k'})=\overline{\nabla}(\mathcal{A})$. Indeed suppose, contrary to our claim, that there are $g\in\overline{\nabla}(\mathcal{A})$, $g'\in\overline{\nabla}(\Delta(\mathcal{A}_{k'}))$ such that $g'\subsetneq g$. Since $\overline{\nabla}(\mathcal{A})$ is formed by elements of cardinality less or equal to $k$ and $\overline{\nabla}(\Delta(\mathcal{A}_{k'}))$ of elements whose cardinality is $k'\ge k$ we have the contradiction $k\ge |g|> |g'|\ge k$ and this concludes the proof of (\ref{eq: nabla dec of alpha - first case}). We now prove the inequality in the statement of the lemma. By (\ref{eq: disjoint antichain con nabla case - 1}) and inclusion (\ref{eq: nabla dec of alpha - first case}) we have $2|\alpha^{-}(\mathcal{A})|+|\overline{\nabla}(\alpha^{-}(\mathcal{A}))| = 2|\mathcal{A}| -2|\mathcal{A}_{k'}| +2|\Delta(\mathcal{A}_{k'})| + |\overline{\nabla}(\alpha^{-}(\mathcal{A}))|\ge 2|\mathcal{A}| +2(|\Delta(\mathcal{A}_{k'})|-|\mathcal{A}_{k'}|) + |\overline{\nabla}(\mathcal{A})|$. Since $k'\ge \frac{n}{2}+1$, then by \cite[Corollary 2.1.2]{Anderson} we have $|\Delta(\mathcal{A}_{k'})|-|\mathcal{A}_{k'}|\ge 0$, whence $2|\alpha^{-}(\mathcal{A})|+|\overline{\nabla}(\alpha^{-}(\mathcal{A}))| \ge 2|\mathcal{A}| + |\overline{\nabla}(\mathcal{A})|$.
  \\
  Consider now the other case $k'< k$ and $k>\frac{n}{2}+1$. Therefore, by Lemma \ref{lem: well defined}, $\alpha^-(\mathcal{A}) = (\mathcal{A}\setminus\mathcal{A}_{k-1})\cup \Delta(\mathcal{A}_{k-1})$ is an antichain with:
  \begin{equation}\label{eq: disjoint antichain con nabla case - 2}
  (\mathcal{A}\setminus\mathcal{A}_{k-1})\cap \Delta(\mathcal{A}_{k-1})=\emptyset
  \end{equation}
  We now prove the inequality of the statement. We first prove the following inclusion:
  \begin{equation}\label{eq: pseudodecomposition nabla case 2}
    \overline{\nabla}((\mathcal{A}\setminus\mathcal{A}_{k-1})\cup \Delta(\mathcal{A}_{k-1}))\supseteq \overline{\nabla}(\mathcal{A})\setminus \overline{\nabla}(\mathcal{A})_k\cup \Delta(\overline{\nabla}(\mathcal{A})_k)
  \end{equation}
  with $\overline{\nabla}(\mathcal{A})\setminus \overline{\nabla}(\mathcal{A})_k\cap\Delta(\overline{\nabla}(\mathcal{A})_k)=\emptyset$. Let us prove first this last property. Suppose on the contrary that there is an $h\in \overline{\nabla}(\mathcal{A})\setminus \overline{\nabla}(\mathcal{A})_k\cap\Delta(\overline{\nabla}(\mathcal{A})_k)\neq\emptyset$.
  Since $\mathcal{A}$ is augmentable and $k'<k$, then by property 2) of Lemma \ref{lem:augmentable antichains} we have
  \begin{equation}\label{eq: nabla contained in the antichain}
    \Delta(\overline{\nabla}(\mathcal{A})_k)\subseteq \mathcal{A}_{k-1}
  \end{equation}
  Therefore we have $h\in \mathcal{A}_{k-1}\cap \overline{\nabla}(\mathcal{A})\subseteq \mathcal{A}\cap \overline{\nabla}(\mathcal{A})=\emptyset$, a contradiction. Hence the two terms in the right part of the inclusion (\ref{eq: pseudodecomposition nabla case 2}) are disjoint. We divide the proof of (\ref{eq: pseudodecomposition nabla case 2}) into two cases. We first prove $\overline{\nabla}(\mathcal{A})\setminus \overline{\nabla}(\mathcal{A})_k\subseteq \overline{\nabla}((\mathcal{A}\setminus\mathcal{A}_{k-1})\cup \Delta(\mathcal{A}_{k-1}))$ and then $\Delta(\overline{\nabla}(\mathcal{A})_k)\subseteq \overline{\nabla}((\mathcal{A}\setminus\mathcal{A}_{k-1})\cup \Delta(\mathcal{A}_{k-1}))$.
  \begin{itemize}
    \item Case $\overline{\nabla}(\mathcal{A})\setminus \overline{\nabla}(\mathcal{A})_k\subseteq \overline{\nabla}((\mathcal{A}\setminus\mathcal{A}_{k-1})\cup \Delta(\mathcal{A}_{k-1}))$.   It is evident that $\overline{\nabla}(\mathcal{A})\setminus \overline{\nabla}(\mathcal{A})_k\subseteq \overline{\nabla}(\mathcal{A}\setminus\mathcal{A}_{k-1})$, thus it is sufficient to show $\overline{\nabla}(\mathcal{A}\setminus\mathcal{A}_{k-1})\subseteq \overline{\nabla}((\mathcal{A}\setminus\mathcal{A}_{k-1})\cup \Delta(\mathcal{A}_{k-1}))$ and to prove this inclusion we use property \ref{nabla prop 1}) of Lemma \ref{lem: properties overline nabla}. Indeed consider $g\in \overline{\nabla}(\mathcal{A}\setminus\mathcal{A}_{k-1})$ and $g'\in\overline{\nabla}(\Delta(\mathcal{A}_{k-1}))$, then $|g|\le k-1$, $|g'|= k-1$. Therefore the inclusion $g'\subsetneq g$ can not occur and so $\overline{\nabla}(\mathcal{A})\setminus \overline{\nabla}(\mathcal{A})_k\subseteq \overline{\nabla}(\mathcal{A}\setminus\mathcal{A}_{k-1})\subseteq \overline{\nabla}((\mathcal{A}\setminus\mathcal{A}_{k-1})\cup \Delta(\mathcal{A}_{k-1}))$.
    \item Case $\Delta(\overline{\nabla}(\mathcal{A})_k)\subseteq \overline{\nabla}((\mathcal{A}\setminus\mathcal{A}_{k-1})\cup \Delta(\mathcal{A}_{k-1}))$. Using (\ref{eq: nabla contained in the antichain}) and property \ref{nabla prop 6}) of Lemma \ref{lem: properties overline nabla} we have
  $$
  \Delta(\overline{\nabla}(\mathcal{A})_k)\subseteq \mathcal{A}_{k-1}\subseteq \nabla(\Delta(\mathcal{A}_{k-1}))\subseteq \nabla((\mathcal{A}\setminus\mathcal{A}_{k-1})\cup \Delta(\mathcal{A}_{k-1}))
  $$
  To reach a contradiction suppose that there is a $g\in \Delta(\overline{\nabla}(\mathcal{A})_k)$ such that $g\in \nabla((\mathcal{A}\setminus\mathcal{A}_{k-1})\cup \Delta(\mathcal{A}_{k-1}))\setminus \overline{\nabla}((\mathcal{A}\setminus\mathcal{A}_{k-1})\cup \Delta(\mathcal{A}_{k-1}))$. By property \ref{nabla prop 5}) of Lemma \ref{lem: properties overline nabla} there is a $g'\in \nabla((\mathcal{A}\setminus\mathcal{A}_{k-1})\cup \Delta(\mathcal{A}_{k-1}))= \nabla(\mathcal{A}\setminus\mathcal{A}_{k-1})\cup \nabla(\Delta(\mathcal{A}_{k-1}))$ with $g'\subsetneq g$. We consider two cases, either $g'\in \nabla(\mathcal{A}\setminus\mathcal{A}_{k-1})$ or $g'\in \nabla(\Delta(\mathcal{A}_{k-1}))$. If $g'\in \nabla(\mathcal{A}\setminus\mathcal{A}_{k-1})$, then there is a $g''\in \mathcal{A}\setminus\mathcal{A}_{k-1}$ such that $g''\subsetneq g'\subsetneq g$, a contradiction since $g''\in \mathcal{A}$, $g\in \Delta(\overline{\nabla}(\mathcal{A})_k)\subseteq \mathcal{A}_{k-1}\subseteq\mathcal{A}$ and $\mathcal{A}$ is an antichain. On the other hand, if $g'\in \nabla(\Delta(\mathcal{A}_{k-1}))$ then $|g'|=k-1$, moreover, since $g\in \Delta(\overline{\nabla}(\mathcal{A})_k)\subseteq \mathcal{A}_{k-1}$, then $|g|=k-1$ which contradicts $g'\subsetneq g$ and this concludes the proof of inclusion (\ref{eq: pseudodecomposition nabla case 2}).
  \end{itemize}
 We can now conclude the proof of the lemma showing the inequality in the statement. By (\ref{eq: disjoint antichain con nabla case - 2}) and (\ref{eq: pseudodecomposition nabla case 2}) we have
  \begin{eqnarray*}
    2|\alpha^{-}(\mathcal{A})|+|\overline{\nabla}(\alpha^{-}(\mathcal{A}))| &=& 2|\mathcal{A}|+ 2( |\Delta(\mathcal{A}_{k-1})| - |\mathcal{A}_{k-1}|) + |\overline{\nabla}(\alpha^{-}(\mathcal{A}))|\\
    &\ge &  2|\mathcal{A}|+ |\overline{\nabla}(\mathcal{A})| + 2( |\Delta(\mathcal{A}_{k-1})| - |\mathcal{A}_{k-1}|) + \\
    & + & (|\Delta(\overline{\nabla}(\mathcal{A})_k)|-|\overline{\nabla}(\mathcal{A})_k|)
  \end{eqnarray*}
  Since $k>\frac{n}{2}+1$ we have by \cite[Corollary 2.1.2]{Anderson}
  $$
  |\Delta(\mathcal{A}_{k-1})|-|\mathcal{A}_{k-1}|\ge 0,\;|\Delta(\overline{\nabla}(\mathcal{A})_k)|-|\overline{\nabla}(\mathcal{A})_k|\ge 0
  $$
  from which it follows $2|\alpha^{-}(\mathcal{A})|+|\overline{\nabla}(\alpha^{-}(\mathcal{A}))| \ge 2|\mathcal{A}|+ |\overline{\nabla}(\mathcal{A})|$ and this concludes the proof of the lemma.
\end{proof}
We are now in position to prove Theorem \ref{theo: antichain + nabla bound}.
\begin{proof}[Proof of Theorem \ref{theo: antichain + nabla bound}]
  The bound is clearly attained when the antichain consists of all the $\frac{n}{2}$-subsets of $X$. Let us now prove the bound.
  Starting from $\mathcal{A}_0=\mathcal{A}$ by Lemma \ref{lem:augmentable antichains} we suppose without loos of generality that $\mathcal{A}_0$ is augmentable, then applying for instance the upward-augmenting map we obtain a new antichain $\mathcal{A}_1$ for which, by Lemmas \ref{lem: alpha +}, $2|\mathcal{A}_0|+ |\overline{\nabla}(\mathcal{A}_0)|\le 2|\mathcal{A}_1|+|\overline{\nabla}(\mathcal{A}_1)|$
  and $\parallel \mathcal{A}_1\parallel_m> \parallel \mathcal{A}_0\parallel_m$. Furthermore by Lemma \ref{lem:augmentable antichains} we can suppose that $\mathcal{A}_1$ is also augmentable. In this way, by a repeated application of Lemmas \ref{lem:augmentable antichains}, \ref{lem: alpha +}, \ref{lem: alpha -} we can find a sequence of augmentable antichains $\mathcal{A}_i$ such that $2|\mathcal{A}_{i-1}|+ |\overline{\nabla}(\mathcal{A}_{i-1})|\le 2|\mathcal{A}_i|+|\overline{\nabla}(\mathcal{A}_i)|$ and either $\parallel \mathcal{A}_{i}\parallel_m> \parallel \mathcal{A}_{i-1}\parallel_m$ and $\parallel \mathcal{A}_i\parallel_M = \parallel \mathcal{A}_{i-1}\parallel_M$, or $\parallel \mathcal{A}_i\parallel_M<\parallel \mathcal{A}_{i-1}\parallel_M$ and $\parallel \mathcal{A}_{i}\parallel_m= \parallel \mathcal{A}_{i-1}\parallel_m$. This process stops when it is reached an augmentable antichain $\mathcal{A}_j$ with $\frac{n}{2}\ge \parallel \mathcal{A}_j\parallel_M\ge \parallel \mathcal{A}_j\parallel_m \ge \frac{n}{2}-1$. If $\parallel \mathcal{A}_j\parallel_M = \parallel \mathcal{A}_j\parallel_m$, $\mathcal{A}_j$ consists of either $\frac{n}{2}$-subsets or $\frac{n}{2}-1$-subsets and the statement of the theorem clearly holds. Thus we can assume $\parallel \mathcal{A}_j\parallel_M > \parallel \mathcal{A}_j\parallel_m$ and let $\mathcal{B}_1=Min(\mathcal{A}_j)$, $\mathcal{B}_2=Max(\mathcal{A}_j)$. Since $\mathcal{A}_j$ is augmentable, by property 1) of by Lemma \ref{lem:augmentable antichains}, $\nabla(\mathcal{B}_1)\subseteq\overline{\nabla}{\mathcal{A}_j}$. Thus, putting $\mathcal{C}_1=\nabla(\mathcal{B}_1)$, we can decompose $\overline{\nabla}{\mathcal{A}_j}=\mathcal{C}_1\cup\mathcal{C}_2$ where $\mathcal{C}_2\subseteq \nabla{\mathcal{B}_2}$. Let $b_i=|\mathcal{B}_i|$, $c_i=|\mathcal{C}_i|$, for $i=1,2$. Since $\overline{\nabla}(\mathcal{A}_j)\cap\mathcal{A}_j=\emptyset $, then $\mathcal{C}_1\cap\mathcal{B}_2=\emptyset$, moreover since both $\mathcal{B}_2$ and $\mathcal{C}_1$ are $\frac{n}{2}$-subsets of $X$ we get $b_2+c_1\le {n\choose \frac{n}{2}}$. Furthermore, since $\mathcal{A}_j$ is an antichain we also get $b_1+b_2\le {n\choose \frac{n}{2}}$. Hence, since $2|\mathcal{A}|+|\overline{\nabla}(\mathcal{A})|=2(b_1+b_2)+(c_1+c_2)$, we have: $2|\mathcal{A}|+|\overline{\nabla}(\mathcal{A})|\le 2{n\choose n/2} +(b_1+c_2)$. Thus to prove the theorem, it is enough to show $b_1+c_2\le {n\choose n/2 + 1}$. Note that $\mathcal{B}_1$ is formed by $\frac{n}{2}-1$-subsets of $X$, while the elements of $\mathcal{C}_2$ are $\frac{n}{2}+1$-subsets. We claim that $\mathcal{B}_1\cup\mathcal{C}_2$ is an antichain. Indeed, if there is a $z\in \mathcal{B}_1$ and $z'\in \mathcal{C}_2$ with $z\subsetneq z'$, then, since $|z'|=\frac{n}{2}+1$ and $|z|=\frac{n}{2}-1$ there is a $a\in X$ with $z\cup\{a\}\subsetneq z'$. However $z\cup\{a\}\in\nabla{\mathcal{B}_1}=\mathcal{C}_1$ and $z'\in \mathcal{C}_2$ contradict the fact that $\overline{\nabla}{\mathcal{A}_j}$ is an antichain. Therefore $\mathcal{B}_1\cup\mathcal{C}_2$ is an antichain. Let $A_1,\ldots, A_{\ell}$
  be a symmetric chains decomposition of the set of subsets of $X$ (see \cite[Section 3.2]{Anderson}). We define the map $\varphi:\mathcal{B}_1\rightarrow 2^X$ which associates to each $z\in \mathcal{B}_1$ with $z\in A_i$, for some $i\in\{1,\ldots\ell\}$, the ``specular'' set $\varphi(z)$ in $A_i$ with $|z|+|\varphi(z)|=n$. Note that $\varphi$ is clearly injective, furthermore it sends $\frac{n}{2}-1$-subsets into $\frac{n}{2}+1$-subsets. Thus, to prove $b_1+c_2\le {n\choose \frac{n}{2} + 1}$, it is enough to show $\varphi(\mathcal{B}_1)\cap\mathcal{C}_2=\emptyset$. Indeed, suppose, contrary to our claim, that there is $z'\in \varphi(\mathcal{B}_1)\cap\mathcal{C}_2$, then we can find a $z\in \mathcal{B}_1$ with $\varphi(z)=z'$. Since $z,z'$ belong to the same symmetric chain, we get $z\subsetneq z'$ which contradicts the fact that $\mathcal{B}_1\cup\mathcal{C}_2$ is an antichain and this concludes the proof of the theorem.
\end{proof}

\section*{Acknowledgements}
The author acknowledges the support of the Centro de Matem\'{a}tica da Universidade do Porto funded by the European Regional Development Fund through the programme COMPETE and by the Portuguese Government through the FCT under the project PEst-C/MAT/UI0144/2011 and the support of the FCT project SFRH/BPD/65428/2009.

\end{document}